\documentclass[11pt]{amsart}

\usepackage{enumerate,rotating,multind,amsthm,longtable,dsfont,pstricks}
\usepackage{latexsym,amsfonts,amsmath,amssymb,amscd,wrapfig,graphicx,curves}

\newtheorem{theorem}{Theorem}[section]
\newtheorem{lemma}[theorem]{Lemma}
\newtheorem{proposition}[theorem]{Proposition}
\newtheorem{corollary}[theorem]{Corollary}

\theoremstyle{definition}
\newtheorem{definition}[theorem]{Definition}

\theoremstyle{remark}
\newtheorem{remark}[theorem]{Remark}

\numberwithin{equation}{section}

\begin{document}
\title{Unique geodesics for Thompson's metric}

\author{Bas Lemmens}
\address{School of Mathematics, Statistics \& Actuarial Science, Cornwallis Building, 
University of Kent, Canterbury, Kent CT2 7NF, UK. Tel: +44--1227823651, Fax: +44--1227827932\\ 
}
\curraddr{}
\email{B.Lemmens@kent.ac.uk}
\thanks{Bas Lemmens was supported by EPSRC grant EP/J008508/1}

\author{Mark Roelands}
\address{School of Mathematics, Statistics \& Actuarial Science, Cornwallis Building, 
University of Kent, Canterbury, Kent CT2 7NF, UK}
\curraddr{}
\email{mark.roelands@gmail.com}

\subjclass[2010]{Primary 53C22; Secondary 51Fxx, 53C60}

\keywords{Geodesics, Thompson's (part) metric, Hilbert's (projective) metric, cones, isometries}

\date{}
\begin{abstract}
In this paper a geometric characterization of the unique geodesics in  Thompson's metric spaces is presented. This characterization is used to prove a variety of other geometric results. Firstly, it will be shown that there exists a unique Thompson's metric geodesic connecting  $x$ and $y$ in the  cone of positive self-adjoint elements in a unital $C^*$-algebra if, and only if, the spectrum of $x^{-1/2}yx^{-1/2}$ is contained in $\{1/\beta,\beta\}$ for some $\beta\geq 1$. A similar result will be established for symmetric cones. 
Secondly, it will be shown that  if $C^\circ$ is the interior of a finite-dimensional closed cone $C$, then the Thompson's metric space $(C^\circ,d_C)$ can be quasi-isometrically embedded into a finite-dimensional normed space if, and only if, $C$ is a polyhedral cone. Moreover, $(C^\circ,d_C)$ is isometric to a finite-dimensional normed  space if, and only if, $C$ is a simplicial cone. It will also be shown that if $C^\circ$ is the interior of a strictly convex cone $C$ with $3\leq \dim C<\infty$, then every Thompson's metric isometry is projectively linear.   
\end{abstract}

\maketitle
\section{Introduction}
In \cite{Bi} Birkhoff showed that one can use Hilbert's (projective) metric and the contraction mapping principle to prove the existence and uniqueness of a positive eigenvector for a large class of linear operators that leave a closed cone $C$ in a 
Banach space invariant. 
An alternative to Hilbert's metric was introduced by Thompson in \cite{Tho}. Thompson's (part) metric, denoted here by $d_C$, has the advantage that it is a metric on each part of a cone $C$ rather than a metric between pairs of rays in each part. It has found numerous applications in the analysis of linear and nonlinear operators on cones, see for instance \cite{AGLN,HIR,LNBook,Nmem1} and the references therein.  Thompson's metric is also used to study the geometry of cones of positive operators \cite{ACS,CM,CPR,LW} and symmetric cones \cite{LL1,Lim1,Lim2,LP}, where it provides an alternative to the usual Riemannian metric. It also appears in the analysis of order-isomorphisms on cones, see \cite{NS1,NS2}.

Despite the frequent use  of Thompson's metric spaces in mathematical analysis, there are still many interesting aspects of their geometry  that remain to be explored. A number of individual results exist. For example, it is known that Thompson metric spaces are Finsler manifolds, see \cite{Nu}. Furthermore, on the cones of positive self-adjoint elements in unital $C^*$-algebras and symmetric cones, Thompson's metric possesses  certain non-positive curvature properties, see \cite{ACS,LL1}. On general closed cones Thompson's metric is semi-hyperbolic, see \cite{NW}. It is also known \cite[Section 2.2]{LNBook} that if $C^\circ$ is the interior of a closed polyhedral cone in a vector space $V$, then $(C^\circ, d_C)$ can be isometrically embedded into $(\mathbb{R}^m,\|\cdot\|_\infty)$, where $\|z\|_\infty = \max_i |z_i|$ is the {\em sup-norm} and $m$ is the number of facets of $C$. Moreover, if $C$ is an $n$-dimensional {\em simplicial cone} in $V$, that is to say, there exist linearly independent vectors $v_1,\ldots,v_n\in V$ such that 
$
C=\{\sum_i \lambda_i v_i\colon \lambda_i\geq 0\mbox{ for all }i\}$,   
then $(C^\circ,d_C)$ is isometric to $(\mathbb{R}^n,\|\cdot\|_\infty)$. 
Furthermore if $\Lambda_{n+1}=\{(s,x)\in \mathbb{R}\times \mathbb{R}^n\colon s^2-x_1^2-\cdots -x_n^2\geq 0\mbox{ and }s\geq 0\}$ is the {\em Lorentz cone}, then $(\Lambda_{n+1}^\circ, d_{\Lambda_{n+1}})$ contains an isometric copy of the real $n$-dimensional hyperbolic space. In fact, on the upper sheet of the hyperboloid $H =\{(s,x)\in \mathbb{R}\times\mathbb{R}^n\colon s^2 - x_1^2 -\cdots - x_n^2=1\}$, Thompson's metric coincides with the hyperbolic distance, see \cite{Lim2} or \cite[Section 2.3]{LNBook}. 

One of the main objectives of this paper is to give a geometric characterization of the unique geodesics in Thompson's metric spaces. This characterization is subsequently used to prove a variety of other results. 

In particular, we show in Section 5 that if $A_+^\circ$ is the interior of the cone of positive self-adjoint elements in a unital $C^*$-algebra $A$, then there exists a unique Thompson metric geodesic connecting $x$ and $y$ in $A_+^\circ$ if, and only if, $\sigma(x^{-1/2}yx^{-1/2})\subseteq \{\beta,1/\beta\}$ for some $\beta\geq 1$. Here $\sigma(z)$ denotes the spectrum of $z$. It turns out that a similar result holds for elements in a symmetric cone. In fact, we will prove in Section 6 that there exists a unique Thompson metric geodesic connecting $x$ and $y$ in a symmetric cone if, and only if, $\sigma(P(y^{-1/2}) x)\subseteq \{\beta,1/\beta\}$ for some $\beta\geq 1$.  Here $P$ is the quadratic representation. These results generalize \cite[Theorem 5.2]{Lim3} by Lim, who showed the equivalence for the cone of positive definite Hermitian matrices. 

The characterization will also be used to prove a number of geometric properties of Thompson's metric spaces.  For example we prove in Section 7 that  if $C$ is a finite-dimensional closed cone with nonempty interior, then $(C^\circ,d_C)$ can be quasi-isometrically embedded into  a finite-dimensional normed space if, and only if,  $C$ is a polyhedral cone.  Furthermore we show that a Thompson's metric space $(C^\circ,d_C)$ is isometric to an $n$-dimensional normed space if, and only if, $C$ is an $n$-dimensional simplicial cone.  Analogous results for Hilbert's metric spaces were obtained by Colbois and Verovic \cite{CV}, and  by Foertsch and Karlsson \cite{FoK}, see also \cite{Ber}. Our method of proof is similar to theirs, but interesting adaptations need to be made to make the arguments works. 

In the final section it will  be shown that if $C$ is a strictly convex cone with nonempty interior and $3\leq \dim C<\infty$, then every isometry of $(C^\circ,d_C)$ is projectively linear. This result complements recent work by Bosch\'e \cite{Bo} who determined the isometries for Thompson's metric on  symmetric cones, and work by Moln\'ar \cite{Mo} on Thompson's metric isometries on the cone of  positive self-adjoint operators on a Hilbert space. In \cite{dlH} de la Harpe proved a similar result for strictly convex  Hilbert's metric spaces. Our proof will appeal to his result. 

In the next section we  recall some basic concepts and results. 

\section{Thompson's metric}
Let $C$ be a  cone in a vector space $V$. So, $C$ is convex, $\lambda C\subseteq C$ for all $\lambda \geq 0$, and $C\cap(-C)=\{0\}$.  The cone $C$ induces a partial ordering $\leq_C$ on $V$ by $x\leq_C y$ if $y-x\in C$. For $x,y\in C$, we say that $y$ {\em dominates} $x$ if there exists $\beta>0$ such that $x\leq_C\beta y$. Given $x,y\in C$ we write $x\sim_C y$ if $y$ dominates $x$, and $x$ dominates $y$. In other words, $x\sim_C y$ if and only if there exist $0<\alpha\leq\beta$ such that $\alpha y\leq_C x\leq_C \beta y$. It is easy to verify that $\sim_C$ is an equivalence relation on $C$. The equivalence classes are called {\em parts} of $C$.  If $C$ is a finite-dimensional closed cone, then the parts are precisely the relative interiors of the faces of $C$, see \cite[Lemma 1.2.2]{LNBook}. Recall that a nonempty convex set $F\subseteq C$ is a {\em face} of $C$ if  $x,y\in C$ and $\lambda x+(1-\lambda)y\in F$ for some $0<\lambda<1$ implies $x,y\in F$. 
The {\em relative interior} of a convex set $S\subset V$ is its interior in the affine span of $S$. 

Given $x,y\in C$ such that $x\sim_C y$, we define 
\[
M(x/y;C) =\inf\{\beta>0\colon x\leq_C \beta y\} \mbox{ and } 
m(x/y;C) =\sup\{\alpha>0\colon \alpha y\leq_C x\}. 
\]
We simply write $M(x/y)$ and $m(x/y)$ if $C$ is clear from the context. Note that 
$m(x/y)=M(y/x)^{-1}$. 

\begin{definition}\label{Thompson}
On a  cone $C$ in a vector space $V$, {\em Thompson's metric}, $d_C\colon C\times C\to [0,\infty]$, is defined by 
\[
d_C(x,y)=\log\Big(\max\{M(x/y),M(y/x)\}\Big)
\]
for $x\sim_C y$ in $C$, and $d_C(x,y) =\infty$ otherwise. 
\end{definition}
This metric was introduced by Thompson in \cite{Tho}, who showed that $d_C$ is a metric on each part of $C$, when $C$ is a closed cone in a normed space. Furthermore, he showed that if $C$ is a closed cone in a Banach space $(V,\|\cdot\|)$, and $C$ is a normal cone,  i.e., there exists $\kappa>0$ such that $\|x\|\leq \kappa\|y\|$ whenever $x\leq_C y$, then 
$(P,d_C)$ is a complete metric space for each part $P$ of $C$, and the topology coincides with the norm topology on $P$. In particular, Thompson's metric  topology on the interior of a closed finite-dimensional cone coincides with the norm topology. 

It can be shown, see \cite[Appendix A.2]{LNBook}, that $d_C$ is a genuine metric on each part  if $C$ is an  {\em almost Archimedean} cone, i.e., if $x\in V$ and there exists $y\in V$ such that $-\epsilon y\leq_C x\leq_C \epsilon y$ for all $\epsilon>0$, then $x=0$. Almost Archimedean cones can be characterized by their intersections with 
finite-dimensional linear subspaces. To state this result  the following notation is convenient. 

Given an almost Archimedean cone $C$ in a vector space $V$ and $S\subseteq V$, we let $V(S)=\mathrm{span}\{S\}$. If $\dim V(S)<\infty$, then we define $C(S) = \overline{C\cap V(S)}$,  where the topology is the unique topology that turns $V(S)$ into a Hausdorff topological vector space. We denote the interior of $C(S)$ in $V(S)$ by $C(S)^\circ$, and its boundary in $V(S)$ by $\partial C(S)$.  Now the characterization of almost Archimedean cones can be stated as follows. 
\begin{lemma}\label{lem:aa}
A cone $C$ in a vector space $V$ is almost Archimedean if and only if for each finite dimensional subspace $W$ of $V$ we have that $C(W)$ is a cone. 
\end{lemma}
\begin{proof}
From \cite[Proposition A.2.2]{LNBook} we know that $C$ is almost Archimedean if and only if for each 2-dimensional subspace $W$ of $V$ we have that $C(W)$ is a cone. Thus, it remains to show that the condition is necessary. So, let $C$ be an almost Archimedean cone and let $W$ be a finite-dimensional subspace of $V$. We need to show that $C(W)$ is a cone.  It is clear that $C(W)$ is convex and $\lambda C(W)\subseteq C(W)$ for all $\lambda\geq 0$. Suppose that there exists $x\neq 0$ such that $x$ and $-x$ in $C(W)$. Note that we can replace $W$ by $C(W)-C(W)$ and assume that the span of $C(W)$ is $W$. As $W$ is finite-dimensional, this implies that 
$C(W)^\circ$ is nonempty. 

Select $y\in C(W)^\circ$ and $\delta>0$ such that $B_\delta(y)\subseteq C(W)^\circ$, where $B_\delta(w)$ denotes the $\delta$-ball around  $w$ in $W$. Let $\epsilon>0$. There exists $z\in B_\delta(0)\cap C(W)$ such that $x+\epsilon z\in C(W)$. Using the convexity of $C(W)$ we see that 
$\frac{1}{1+\epsilon} x +(1-\frac{1}{1+\epsilon})y = \frac{1}{1+\epsilon}(x+\epsilon z) + 
(1+\frac{1}{1+\epsilon})(y-z)\in C(W)$. These points lie in $\mathrm{span}\{x,y\}$ and converge to $x$ as $\epsilon \to 0$. In the same way we can find points in $\mathrm{span}\{x,y\}$ converging to $-x$. This implies that $x$ and $-x$ are in $C(x,y)$, which is impossible by \cite[Proposition A.2.2]{LNBook}.
\end{proof}

A useful variant of Thompson's metric, which will also play a role here,  is {\em Hilbert's (projective) metric},
\[
\delta_C(x,y) =\log \Big{(}M(x/y)M(y/x)\Big{)}
\]
for $x\sim_C y$ in $C$, and $d_C(x,y) =\infty$ otherwise.  Hilbert's metric is only  a metric on the rays in each part of $C$, as $\delta_C(\lambda x,\mu y)=\delta_C(x,y)$ for all $\lambda,\mu>0$ and $x\sim_Cy$  in $C$.  

Given a cone $C$ in $V$ we denote  the {\em dual cone} by $C^*=\{\phi\in V^*\colon \phi(x)\geq 0\mbox{ for all }x\in C\}$. A linear functional $\phi\in C^*$ is said to be {\em strictly positive} if $\phi(x)>0$ for all $x\in C\setminus\{0\}$.  
It is well know, see for example, \cite[Theorem 2.1.2]{LNBook}, that if $C$ is a closed cone  with nonempty interior in a finite-dimensional vector space $V$, then $C^*$ is also a closed cone with nonempty interior. Moreover, for each strictly positive $\varphi\in C^*$  the set  $\Sigma^\circ_\varphi=\{x\in C^\circ\colon \varphi(x)=1\}$ is a  bounded convex set on which $\delta_C$ coincides with {\em Hilbert's cross-ratio metric}, 
\[
\kappa(x,y) =\log \Big{(}\frac{\|x'-y\|}{\|x'-x\|}\frac{\|y'-x\|}{\|y'-y\|}\Big{)},
\]
where $x'$ and $y'$ are the points of intersection of the straight line through $x$ and $y$ and $\partial \Sigma_\varphi^\circ$ such that $x$ is between $x'$ and $y$, and $y$ is between $y'$ and $x$.

We will be interested in the geodesics in $(C,d_C)$. Recall that a map $\gamma$ from an (open, closed, bounded, or, unbounded) interval $I\subseteq \mathbb{R}$ into a metric space $(X,d_X)$ is called a {\em geodesic path} if 
\[
d_X(\gamma(s),\gamma(t)) = |s-t|\mbox{\quad for all }s,t\in I.
\]
The image of $\gamma$ is called a {\em geodesic segment} in $(X,d_X)$. It said to be a {\em geodesic line} in $(X,d_X)$ if $I=\mathbb{R}$.  

It is known, see for example  \cite[Theorem 2.6.9]{LNBook},  that if $P$ is a part of $C$, then  $(P,d_C)$ is a {\em geodesic} metric space, i.e.,  for each $x, y\in P$ there exists a geodesic path $\gamma\colon [a,b]\to P$ with $\gamma(a) = x$ and $\gamma(b) =y$. 
In general there can be more than one geodesic segment connecting $x$ and $y$ in $(P,d_C)$. One of the main objectives is to characterize those $x$ and $y$ in $(C,d_C)$ that are connected by a {\bf unique} geodesic segment. 
The following elementary result will be useful. We leave the proof to the reader. 
\begin{lemma}\label{lem:2.2}
If $x$ and $y$ are distinct points in a geodesic metric space $(X,d_X)$ and $\gamma\colon[a,b]\to X$ is a geodesic path with $\gamma(a)=x$ and $\gamma(b)=y$, then the image of $\gamma$ is a unique geodesic segment connecting $x$ and $y$ if and only if for each $z\in X$ with $d_X(x,y)=d_X(x,z)+d_X(z,y)$, we have that $z=\gamma(t)$ for some $t\in I$.     
\end{lemma}

\section{Two dimensional cones} 
The following elementary lemma is useful.
\begin{lemma}\label{lem:3.00}
Let $C$ be an almost Archimedean cone. If $x\sim_C y$ in $C$, then $x,y\in C(x,y)^\circ$ and $d_C(w,z) = d_{C(x,y)}(w,z)$ for  all $w,z\in C(x,y)^\circ$. 
\end{lemma}
\begin{proof}
The statements are trivial for $x=y=0$. If $x=\mu y$ for some $\mu>0$ and $x\neq 0$, then $C(x,y) =\{\lambda x\colon\lambda \geq 0\}$ and hence $x,y\in C(x,y)^\circ$. 
Obviously, for $w=\alpha x$ and $z=\beta x$ with $0<\alpha\leq \beta$ we have 
$d_C(w,z)=\log \beta/\alpha =d_{C(x,y)}(w,z)$. 

If $x\sim_C y$ are linearly independent, then $C(x,y)$ is a 2-dimensional closed cone in $V(x,y)$. By \cite[Theorem A.5.1]{LNBook} we know that there exists linearly independent vector $u$ and $v$ in $V(x,y)$ such that 
\[
C(x,y) =\{su+tv\colon s,t\geq 0\}.
\] 
It follows that $C(x,y)$ has 4 parts: $\{0\}$, $\{su\colon s>0\}$, $\{tu\colon t>0\}$, and  
$C(x,y)^\circ$. As $x$ and $y$ are linearly independent, $x$ and $y$ must be in 
$C(x,y)^\circ$. Moreover, it follows from \cite[Corollary A.5.2]{LNBook} that 
\[
M(w/z;C)=M(w/z;C\cap V(x,y)) = M(w/z;C(x,y))
\]
for all $w,z\in C(x,y)^\circ$, which proves the final assertion. 
\end{proof} 

Lemma \ref{lem:3.00} has the following basic consequence.  
\begin{corollary}\label{lem:3.0} 
If $x\sim_C y$  are connected by a unique geodesic segment $\gamma$ in $(C,d_C)$, then $\gamma$ lies in $C(x,y)^\circ$ and $\gamma$ is a unique geodesic segment  connecting $x$ and $y$ in $(C(x,y)^\circ,d_{C(x,y)})$. 
\end{corollary}
Thus, we need to first analyze the problem in two dimensions. If $K$ is a closed cone with nonempty interior in a 2-dimensional vector space $W$, then there exists $u,v\in\partial K$ linearly independent such that 
\[
K=\{\alpha u +\beta v\colon \alpha,\beta\geq 0\},
\] 
see \cite[Theorem A.5.1]{LNBook}. Alternatively, there exists linearly independent functionals $\psi_1$ and $\psi_2$ on $W$ such that 
\[
K=\{x\in W\colon \psi_1(x)\geq 0\mbox{ and } \psi_2(x)\geq  0\}.
\]
\begin{lemma}\label{lem:3.1} Let $K\subseteq W$ be a closed  cone with nonempty interior in a 2-dimensional normed space $W$. If $x,y\in K^\circ$, then there exists a unique geodesic segment connecting $x$ and $y$ in $(K^\circ,d_K)$ if and only if either 
\begin{enumerate}[(i)]
\item $M(x/y)=M(y/x)$, or, 
\item $M(x/y)=M(y/x)^{-1}$, in which case $x=\lambda y$ for some $\lambda>0$.
\end{enumerate}
In particular,  through each $x\in K^\circ$ there are precisely two unique geodesics. 
\end{lemma}
\begin{proof}
Define a map $\Psi\colon K^\circ\to\mathbb{R}^2$ by $\Psi(x)=(\log\psi_1(x),\log\psi_2(x))$. Since $x\leq y$ if and only if $\psi_i(x)\le\psi_i(y)$ for $i=1,2$, it follows that 
$$M(x/y)=\max_{i=1,2}\frac{\psi_i(x)}{\psi_i(y)}$$ on $K^\circ$. So, for $x,y\in K^\circ$ the equalities
\begin{align*}
d_K(x,y)=\max_{i=1,2}\left|\log\frac{\psi_i(x)}{\psi_i(y)}\right| =\|\Psi(x)-\Psi(y)\|_\infty,
\end{align*}
hold, where $\|z\|_\infty=\max_i|z_i|$ is the sup-norm. This implies that 
$\Psi$ is an isometry from $(K^\circ,d_K)$ onto $(\mathbb{R}^2,\|\cdot\|_\infty)$. 
In $(\mathbb{R}^2,\|\cdot\|_\infty)$ there are precisely two unique geodesic lines through each point $z$, namely
\[
\ell_I =\{z+t(1,-1)\colon t\in\mathbb{R}\}\mbox{\quad and \quad} 
\ell_{II} =\{z+t(1,1)\colon t\in\mathbb{R}\},
\]
which proves the last assertion of the lemma. 
 
It follows that there exists a unique geodesic segment connecting $x$ and $y$ in $(K^\circ, d_K)$ if and only if either 
\[
\Psi(x)- \Psi(y) =s(1,-1)\mbox{\quad or \quad } \Psi(x) -\Psi(y) =s(1,1)
\]
for some $s\in\mathbb{R}$. The first equality is equivalent to 
\[
\log \frac{\psi_1(x)}{\psi_1(y)} = s= \log \frac{\psi_2(y)}{\psi_2(x)},
\]
which holds if and only if $M(x/y)=M(y/x)$. 
The second equality is equivalent to 
\[
\log \frac{\psi_1(x)}{\psi_1(y)} = s= \log \frac{\psi_2(x)}{\psi_2(y)},
\]
which holds if and only if $M(x/y)=M(y/x)^{-1}$. Finally note that as $K$ is closed, 
$M(y/x)^{-1} y\leq_K x\leq_K M(x/y)y$. So, if $M(x/y)=M(y/x)^{-1}$, then $x=M(x/y)y$. 
 \end{proof}
As an immediate consequence we obtain the following result. 
\begin{corollary}\label{cor:3.2} 
Suppose that $C$ is an almost Archimedean cone in a vector space $V$. If $x\sim_C y$ are linearly independent elements of $C$ and the exists a unique geodesic segment connecting $x$ and $y$ in $(C, d_C)$, then $M(x/y)=M(y/x)$. 
\end{corollary}
It will be convenient to make the following definition. 
\begin{definition}\label{def:type} Let $C$ be an almost Archimedean cone in a vector  space $V$. If $x\sim_C y$ are linearly independent elements in $C$ and $M(x/y)=M(y/x)$, then we call the unique geodesic segment (line) through $x$ and $y$ in $(C(x,y)^\circ, d_{C(x,y)})$ a {\em type I geodesic segment (line)} in $(C,d_C)$. For $x\in C\setminus\{0\}$ we call a segment of the ray, $\{tx\colon t>0\}$, through $x$  a {\em type II geodesic segment} in $(C,d_C)$.
\end{definition}
\begin{remark} Note that if $u$ and $v$ are points on a type I geodesic segment, then $M(u/v)=M(v/u)$.  
\end{remark}
\begin{figure}[h]
\begin{center}
\begin{picture}(200, 100)
\thicklines
   \curve(10,100, 100,50, 190,100)
 \put(100,0){\line(-1,2){50}}

\put(100,0){\line(1,1){100}}
\put(100,0){\line(-1,1){100}}

\put(95,90){$C^\circ$}

\put(70,75){$\mathrm{II}$}
\put(133,65){$\mathrm{I}$}

\put(72,60){$x$}
\put(73,55){\circle*{3.0}}

\put(118,60){$y$}
\put(125,54){\circle*{3.0}}

   \end{picture}
   \caption{Type I and type II geodesic segments}
   \label{fig:3.1}
  \end{center}
  \end{figure}
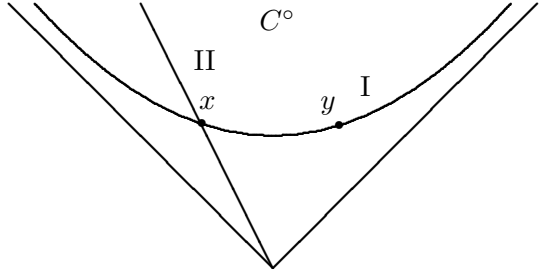

\begin{lemma}\label{lem:3.2} 
Let $K=\{\alpha u +\beta v\in W\colon \alpha,\beta\geq 0\}$  be a closed cone with nonempty interior  in a 2-dimensional vector space $W$. Every type I geodesic line in $(K^\circ,d_K)$ is of the form   
\[
\{\alpha(e^t u+e^{-t}v)\colon t\in\mathbb{R}\}
\]
for some $\alpha>0$. Moreover, for each $\alpha>0$, the map $\gamma\colon t\mapsto \alpha (e^tu+e^{-t}v)$, $t\in\mathbb{R}$, is a geodesic path and its image is a type I geodesic line  in $(K^\circ,d_K)$. 
\end{lemma}
\begin{proof}
Let $x\in K^\circ$. As $u$ and $v$ are linearly independent, there exist  unique $a,b>0$ such that $x=au+bv$. A simple linear algebra argument shows that $a=\alpha e^t$ and $b=\alpha e^{-t}$ has a unique solution  with $\alpha>0$ and $t\in\mathbb{R}$. Thus there exist unique  $\alpha>0$ and $t\in\mathbb{R}$ such that $x=\alpha(e^tu+e^{-t}v)$. 

As $K$ is 2-dimensional, there is exactly one type I geodesic line through $x$ in $(K^\circ,d_K)$. So, it suffices to show for $\alpha>0$ that the image of $\gamma\colon\mathbb{R}\to (K^\circ,d_K)$ given by, 
\[
\gamma(t)=\alpha(e^t u+e^{-t}v)\mbox{\quad for } t\in\mathbb{R},
\] 
is a type I geodesic line. Let $t>s$ and note that 
\[
e^{t-s}\gamma(s)-\gamma(t) =\alpha(e^{t-2s}-e^{-t})v\in\partial K,
\]
so that $M(\gamma(t)/\gamma(s))=e^{t-s}$. Likewise, 
\[
\gamma(t) - e^{s-t} \gamma(s) = \alpha(e^t-e^{2s-t})u\in\partial K
\]
implies that $M(\gamma(s)/\gamma(t)) =e^{t-s}$.  Thus, $d_K(\gamma(t),\gamma(s)) =t-s$ and 
\[
M(\gamma(t)/\gamma(s))=M(\gamma(s)/\gamma(t))\mbox{\quad  for all $t>s$}.
\]
This shows that $\gamma(\mathbb{R})$ is a unique type I geodesic line in $(K^\circ,d_K)$. 
\end{proof}

\section{A characterization of unique geodesics} 
In this section we prove a geometric characterization of the unique geodesic segments  in $(C,d_C)$.  As we shall see, it is quite easy to show that a type II geodesic segment is always a unique geodesics segment in the whole space $(C,d_C)$. In general, however, additional assumptions are needed for a type I geodesic to be unique in the whole space. 

\begin{proposition}\label{prop:4.1} 
Let $C$ be an almost Archimedean cone in a vector space $V$, If $x\in C\setminus\{0\}$ and $y=\lambda x$ for some $\lambda>1$, then the type II geodesic segment, $\{\lambda^tx\colon 0\leq t\leq 1\}$, connecting $x$ and $y$ is a unique geodesic segment in $(C^\circ,d_C)$.
\end{proposition}
\begin{proof}
Suppose that $z\in C$ is such that 
\[ 
d_C(x,z)=sd_C(x,y)\mbox{\quad and \quad } d_C(z,y)=(1-s)d_C(x,y).
\]
As $\lambda >1$, $d_C(x,y) =\log M(y/x) =\log\lambda$. Thus, 
$M(z/x)\leq \lambda^s$ and $M(y/z)\leq\lambda^{(1-s)}$. 
It follows from the first inequality that $z\leq_C (\lambda^s+\epsilon) x$ for all $\epsilon >0$. The second inequality gives  $y\leq_C (\lambda^{1-s} +\epsilon )z$ for all $\epsilon >0$. As $y=\lambda x$ we find that 
\[
\frac{\lambda}{\lambda^{1-s} +\epsilon} x\leq_C z\leq_C (\lambda^s+\epsilon) x
\]
for all $\epsilon >0$. This implies that $z=\lambda^sx$, as $C$ is almost Archimedean. 
\end{proof}
Before we analyze the type I geodesic segments, we prove the following basic lemma. 
\begin{lemma}\label{lem:4.2} 
Let $C$ be an almost Archimedean cone in a vector space $V$. If $x\sim_C y$ are linearly independent elements in $C$ and $M(x/y)=M(y/x)$, then the straight line through $x$ and $y$ intersects $\partial C(x,y)$ in precisely two points. 
\end{lemma}
\begin{proof}
Suppose that $x\sim_C y$ are linearly independent and $M(x/y)=M(y/x)$. 
Write $\beta =M(x/y)$. Note that $d_{C(x,y)}(x,y) =d_C(x,y)=\log \beta>1$ by Lemma \ref{lem:3.00}. So, $x\leq_{C(x,y)}\beta y$ and $y\leq_{C(x,y)}\beta x$, as $C(x,y)$ is 
closed. This implies that $x-\frac{1}{\beta}y\in\partial C(x,y)$ and $y-\frac{1}{\beta}x\in\partial C(x,y)$. Thus, 
\[
x'= \frac{\beta}{\beta -1} x -\frac{1}{\beta -1}y\in\partial C(x,y)
\mbox{\quad and\quad }
y'= \frac{\beta}{\beta -1} y -\frac{1}{\beta -1}x\in\partial C(x,y).
\]
Obviously, $x'$ and $y'$ also lie on the straight line through $x$ and $y$. 
\end{proof}

\begin{theorem}\label{thm:4.3} 
Let $C$ be an almost Archimedean cone in a vector space $V$. Suppose that $x\sim_C y$ are linearly independent elements of $C$ and $M(x/y) =M(y/x)$. Let $x',y'\in\partial C(x,y)$ be the points of intersection of the straight line through $x$ and $y$ such that $x$ is between $x'$ and $y$ and $y$ is between $y'$ and $x$.  The type I geodesic segment connecting $x$ and $y$ is a unique geodesic segment in $(C,d_C)$ if and only if 
there exist no $z\in V\setminus\{0\}$ and $\epsilon >0$ such that $x'+tz\in\partial C(x,y,z)$ and $y'+tz\in\partial C(x,y,z)$ for all $|t|<\epsilon$. 
\end{theorem}
\begin{proof}
Suppose that $z\in V\setminus\{0\}$ and $\epsilon>0$ are such that $x'+tz\in \partial C(x,y,z)$ and $y'+tz\in\partial C(x,y,z)$ whenever $|t|<\epsilon$. Let $\gamma$ be the type I geodesic segment connecting $x$ and $y$. Further let $\zeta$ be the point on $\gamma$  with the property 
\[
d_C(x,\zeta)=\frac{1}{2}d_C(x,y)=d_C(\zeta,y).
\] 
For $\delta>0$ define $\zeta_\delta=\zeta+\delta z$. Note that as $\zeta$ lies on $\gamma$, $M(x/\zeta)=M(\zeta/x)$. By Lemma \ref{lem:4.2} the straight line through $x$ and $\zeta$ intersects $\partial C(x,y)$ in two points $\tilde{x}$ and $\zeta'$, as in Figure \ref{Fig:4.1}.  Note that $\tilde{x}$ is a positive multiple of $x'$ and $\zeta'$ is a positive multiple of $y'$, as $\zeta\in\mathrm{span}\{x,y\}$. 

\begin{figure}[h]
\begin{center}
\thicklines
\begin{picture}(160,120)(0,0)
\put(0,10){\line(1,0){160}} 
\put(0,110){\line(1,0){160}} 
\put(80,10){\line(0,1){100}}
\put(30,10){\line(3,4){75}}
\put(45,30){\line(1,0){35}}

\put(80,10){\circle*{3.0}}\put(80,0){$\zeta'$}
\put(80,30){\circle*{3.0}}\put(85,30){$\zeta$}
\put(80,77){\circle*{3.0}}\put(85,75){$x$}
\put(80,110){\circle*{3.0}}\put(80,115){$\tilde{x}$}
\put(30,10){\circle*{3.0}}\put(30,0){$\zeta'_\delta$}
\put(45,30){\circle*{3.0}}\put(30,30){$\zeta_\delta$}
\put(105,110){\circle*{3.0}}\put(105,115){$x'_\delta$}

\end{picture}
\caption{The points in the boundary\label{Fig:4.1}}
\end{center}
\end{figure}
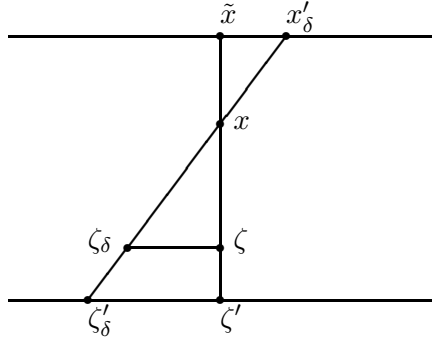

If $s>1$ is such that $sx+(1-s)\zeta=\tilde{x}$, then 
\[
sx+(1-s)\zeta_\delta=sx+(1-s)\zeta+(1-s)\delta z=\tilde{x}+(1-s)\delta z. 
\] 
As $\tilde{x}$ is a multiple of $x'$, $\tilde{x}+\lambda z\in\partial C(x,y,z)$ for all $\lambda\in\mathbb{R}$ with $|\lambda |$ small. 
Thus for all $\delta>0$ sufficiently small  $x'_\delta=\tilde{x}+(1-s)\delta z\in\partial C(x,y,z)$. 
Similarly, if we let $t<0$ be such that $\zeta'= tx+(1-t)\zeta$, then 
\[ 
tx+(1-t)\zeta_\delta= tx +(1-t)\zeta +(1-t)\delta z=\zeta'+(1-t)\delta z.
\]
As $\zeta'$ is a multiple of $y'$, the point $\zeta'_\delta=\zeta'+(1-t)\delta z\in\partial C(x,y,z)$ for all $\delta >0$ small. 
Note also that if $x'_\delta,\zeta'_\delta\in\partial C(x,y,z)$, then  $\zeta_\delta \in C(x,y,z)$. 

Recall that $\tilde{x}=sx -(1-s)\zeta $ and $\zeta'=tx-(1-t)\zeta$. Using  similarity of triangles in  Figure \ref{Fig:4.2}, we see  that 
\begin{figure}[h]
\begin{center}
\thicklines
\begin{picture}(150,100)(0,0)
\put(50,0){\line(-1,2){50}} 
\put(50,0){\line(1,1){100}}
\put(0,80){\line(1,0){150}} \put(30,80){\circle*{3.0}}
\put(30,85){$x$} \put(90,80){\circle*{3.0}}\put(90,85){$\zeta$}
\put(50,0){\line(1,2){40}} 
\put(30,80){\line(-1,-2){10}}
\put(30,80){\line(1,-4){20}}
\put(90,80){\line(1,-4){8}}
\put(98,48){\circle*{3.0}}\put(103,42){$\zeta -M(x/\zeta)^{-1}x$}
\put(20,60){\circle*{3.0}}\put(-60,55){$x-M(\zeta/x)^{-1}\zeta$}\put(10,80){\circle*{3.0}}\put(12,85){$\tilde{x}$}
\put(130,80){\circle*{3.0}}\put(123,85){$\zeta'$} 
\end{picture}
\caption{ Identities \label{Fig:4.2}}
\end{center}
\end{figure}
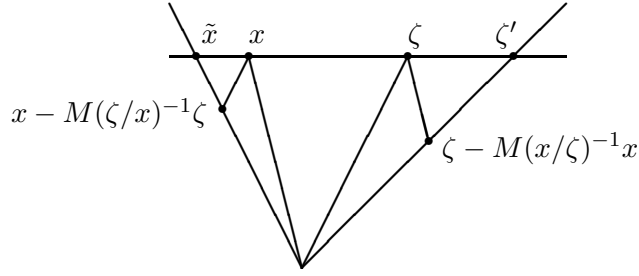
 
\[
M(\zeta/x)=\frac{s}{s-1}
\mbox{\quad and\quad } M(x/\zeta)=\frac{t}{t-1}.
\] 

Since   $x'_\delta=sx -(1-s)\zeta_\delta $ and $\zeta'_\delta=tx-(1-t)\zeta_\delta$, we can derive in the same way that 
\[
M(\zeta_\delta/x)=\frac{s}{s-1}
\mbox{\quad and\quad } M(x/\zeta_\delta)=\frac{t}{t-1}.
\]
This implies that  $d_C(x,\zeta_\delta)=d_C(x,\zeta)=\frac{1}{2}d_C(x,y)$.
Analogously, for $\delta>0$ small enough we  have $d_C(\zeta_\delta,y)=d_C(\zeta,y)=\frac{1}{2}d_C(x,y)$. It now follows from Lemma \ref{lem:2.2} that $\gamma$ is not a unique geodesic segment. 

Conversely, suppose that $\gamma$ is not a unique geodesic segment connecting $x$ and $y$ in $(C,d_C)$. It follows from Lemmas \ref{lem:2.2} and \ref{lem:3.1} that there exists an element $\zeta\in C\setminus C(x,y)^\circ$ such that $d_C(x,\zeta)+d_C(\zeta,y)=d_C(x,y)$. As
\begin{eqnarray*}
d_C(x,y)&= &\log M(x/y)\\
 & \leq & \log (M(x/\zeta ) M(\zeta /y))\\
  & =& \log M(x/\zeta)+\log M(\zeta /y)\\
   & \leq &  d_C(x,\zeta)+d_C(\zeta,y)\\
   &= & d_C(x,y),
\end{eqnarray*}
we have that $d_C(x,\zeta)=\log M(x/\zeta)$ and $d_C(\zeta,y)=\log M(\zeta/y)$. Also, since $M(x/y)=M(y/x)$, we have $d_C(x,\zeta)=\log M(\zeta/x)$ and $d_C(\zeta,y)=\log M(y/\zeta)$. Write $K=C(x,y,\zeta)$. Then $K$ is a 3-dimensional closed cone in $W=\mathrm{span}\{x,y,\zeta\}$, with $x$, $y$ and $\zeta$ in its interior. Let $\varphi\colon W\to\mathbb{R}$ be a strictly positive functional.  Such a functional exists, since $K$ is a finite dimensional closed cone, see \cite[Lemma 1.2.4]{LNBook}. Consider the bounded convex set 
$\Sigma_\varphi^\circ=\{w\in K^\circ:\varphi(w)=1\}$. 
 
Now, for the Hilbert metric $\delta_K$ on $\Sigma_\varphi^\circ$ and the elements $[x]=x/\varphi(x)$, $[y]=y/\varphi(y)$ and $[\zeta]=\zeta/\varphi(\zeta)$, our previous findings together with the scalar invariance of $\delta_K$ imply that
\begin{eqnarray*}
\delta_K([x],[y])&\leq & \delta_K([x],[\zeta])+\delta_K([\zeta],[y])\\
& =& \delta_K(x,\zeta)+\delta_K(\zeta,y)\\
&= & 2d_C(x,\zeta)+2d_C(\zeta,y)\\
& =& 2d_C(x,y)\\
& =& \delta_K(x,y)\\
&=& \delta_K([x],[y]).
\end{eqnarray*}
Straight line segments are geodesic segments in $(\Sigma_\varphi^\circ,\delta_K)$, see  for example \cite[Section 5.6]{Pa}.  So, it follows from the previous equality and  Lemma \ref{lem:2.2} that there exists more than one geodesic segment in $(\Sigma_\varphi^\circ,\delta_K)$ connecting $[x]$ and $[y]$. This implies that there exists two straight line segments $I_x$ and $I_y$ in $\partial \Sigma_\varphi^\circ$ such that the endpoints $u\in\partial K$ and $v\in\partial K$ of the straight line segment through $[x]$ and $[y]$ lie in the relative interiors of $I_x$ and $I_y$, respectively, see for example \cite[Theorem 5.6.7]{Pa}. Thus, $u$ and $v$ lie in the relative interiors of two distinct 2-dimensional faces of $K$. Since $W=\mathrm{span}\{x,y,\zeta\}$ is 3-dimensional, it follows that the intersection of the span of these two faces is non-trivial. To that end, let $z\neq 0$ be a point in the intersection of the spans of the faces of $u$ and $v$ in $W$. Then there exists $\eta>0$ such that 
$u+\mu z\in\partial K$ and $v+\mu z\in\partial K$ whenever $|\mu|<\eta$. 

As $x'=\alpha u$ and $y'=\beta v$ for some $\alpha,\beta>0$, we conclude that there exists an $\epsilon>0$ such that $x'+tz\in\partial K$ and $y'+tz\in\partial K$ whenever 
$|t|<\epsilon$. To finish the proof, it remains to be shown that $K=C(x,y,z)$. To establish this equality, we argue by contradiction that $z\not\in\mathrm{span}\{x,y\}$. 
We know that there exists linearly independent functionals $\psi_1$ and $\psi_2$ on $V(x,y)$ such that 
\[
C(x,y)=\{w\in V(x,y)\colon \psi_1(w)\geq 0\mbox{ and }\psi_2(w)\geq 0\}
\]
and $\psi_1(x')=0=\psi_2(y')$. So, $\psi_1(x'+tz)=t\psi_1(z)\geq 0$ and $\psi_2(y'+tz)=t\psi_2(z)\geq 0$ for all $|t|<\epsilon$. So, $\psi_1(z)=0=\psi_2(z)$, and hence  $z=0$, as $\psi_1$ and $\psi_2$ are linearly independent, which is impossible. 
\end{proof}
Note that if $C$ is a closed cone with nonempty interior in a normed space $V$, then $C^\circ$ is a part of $C$. In that case, if  $x,y\in C^\circ$ and  $M(x/y)=M(y/x)$ then  the type I geodesic segment connecting $x$ and $y$ is unique if and only if there exists no $z\in V\setminus\{0\}$ and $\epsilon >0$ such that $x'+tz\in\partial C$ and $y'+tz\in\partial C$ for all $|t| < \epsilon $.

The type I unique geodesics in $(C, d_C)$ are closely related to  unique Hilbert's metric geodesics as the following lemma shows. 
\begin{lemma}\label{lem:4.4} Let $C$ be an almost Archimedean cone in a  vector space $V$. Suppose that  $\varphi\in V^*$ is a strictly positive functional and let $\Sigma_\varphi=\{x\in C\colon\varphi(x)=1\}$. If $x\sim_C y$ are linearly independent elements in $C$ and $M(x/y)=M(y/x)$, then the type I geodesic connecting $x$ and $y$ is unique in $(C,d_C)$  if and only if the straight line segment connecting $[x]=x/\varphi(x)$ and $[y]=y/\varphi(y)$ is a unique geodesic in $(\Sigma_\varphi\cap P_x,\delta_C)$, where $P_x$ is the part of $x$. 
\end{lemma}
\begin{proof} 
It is known, see \cite[Proposition 1.9]{Nmem1}, that straight lines are geodesic segments in $(\Sigma_\varphi\cap P_x,\delta_C)$. Now suppose that the type I geodesic segment connecting $x$ and $y$ in $(C,d_C)$ is not unique. Then there exists $z\in P_x$ with $z\not\in \mathrm{span}\{x,y\}$ and 
\[
d_C(x,y)=d_C(x,z)+d_C(z,y). 
\]
As $M(x/y)=M(y/x)$, we have 
\begin{eqnarray*}
\log M(x/y) & = & d_C(x,y)\\
  & = & d_C(x,z)+d_C(z,y)\\
  & \geq &\log M(x/z)+\log M(z/y) \\
  &\geq & \log M(x/y),
\end{eqnarray*}
so that $d_T(x,z) =\log M(x/z)$ and $d_T(z,y) = \log M(z/y)$. Using the fact that $\log M(y/x)=d_T(x,y)$, it can be shown in the same way that $d_T(x,z) =\log M(z/x)$ and $d_T(z,y) = \log M(y/z)$. Thus, 
\begin{equation}\label{eq:3a}
M(x/z)=M(z/x)\mbox{\quad and \quad} M(y/z)=M(z/y). 
\end{equation}

Writing $[u]=u/\varphi(u)$ for $u\in C\setminus\{0\}$, it now follows from (\ref{eq:3a}) that 
\[
\delta_C([x],[y])= 2d_C(x,y) = 2d_C(x,z)+2d_C(z,y) = \delta_C([x],[z])+\delta_C([z],[y]).
\] 
As $z\not\in \mathrm{span}\{x,y\}$, $[z]$ is not on the straight line segment connecting  $[x]$ and $[y]$. It now follows from Lemma \ref{lem:2.2} that there is more than one geodesic segment connecting $[x]$ and $[y]$ in $(\Sigma_\varphi\cap P_x,\delta_C)$. 

Conversely, suppose that the straight line segment connecting $[x]$ and $[y]$ is not a unique Hilbert's metric geodesic in $\Sigma_\varphi\cap P_x$. Then there exists $w\in \Sigma_\varphi\cap P_x$ with $w\not\in\mathrm{span}\{x,y\}$ such that 
\begin{equation}\label{eq:3b} 
\delta_C([x],[y])=\delta_C([x],w)+\delta_C(w,[y]).
\end{equation}
Recall that $M(x/y)=M(y/x)$. So, for 
\[
\lambda = M(x/w)^{1/2}M(w/x)^{-1/2}\mbox{\quad and \quad}
\mu = M(y/w)^{1/2}M(w/y)^{-1/2}
\]
we have that 
\begin{equation}\label{eq:3c}
M(x/\lambda w)=M(\lambda w/x)\mbox{\quad and \quad} 
M(y/\mu w)=M(\mu w/y). 
\end{equation}
We will show by contradiction that $\lambda=\mu$. Without loss of generality assume that $\lambda<\mu$. Note that $2d_C(x,\lambda w)=\delta_C([x],w)$ and $2d_C(\mu w,y) = \delta_C(w,[y])$, so that  
\begin{equation}\label{eq:3d} 
d_C(x,\lambda w)+d_C(y,\mu w) = d_C(x,y)
\end{equation}
by (\ref{eq:3b}). As $\lambda<\mu$, $M(\lambda w/\mu w)=\lambda/\mu<1$, and hence  it follows from (\ref{eq:3c}) and (\ref{eq:3d}) that 
\begin{eqnarray*}
\log M(x/y) & \leq  & \log\left ( M(x/\lambda w)M(\lambda w/\mu w)M(\mu w/y)\right )\\
  & < & \log M(x/\lambda w) +\log M(\mu w/y)\\
  & = & d_C(x,\lambda w)+d_C(y,\mu w)\\
  & = & d_C(x,y) \\
  & = & \log M(x/y),
\end{eqnarray*}
which is absurd, and hence $\lambda=\mu$. 

This implies that $d_C(x,y)=d_C(x,\lambda w)+d_C(\lambda w,y)$. As $\lambda w\not\in\mathrm{span}\{x,y\}$ and $(C,d_C)$ contains the type I geodesic segment in $C(x,y)$ connecting $x$ and $y$, it follows from Lemma \ref{lem:2.2} that this type I geodesic segment is not a unique geodesic segment in $(C,d_C)$. 

\end{proof}

\section{Unique geodesics in unital  $C^*$-algebras}
In this section $A$ will denote a unital $C^*$-algebra and $\Re (A)$ will be the real vector space of the self-adjoint elements in $A$. For standard results in the theory of $C^*$-algebras we refer the reader to \cite{Con}. In $\Re(A)$ all elements have real spectra, which yields a closed cone $A_+=\{a\in\Re(A)\colon \sigma(a)\subseteq [0,\infty)\}$, where  $\sigma(a)$ denotes the spectrum of $a$. It is well known that the interior, $A_+^\circ$, of $A_+$ is the set of those $a\in A_+$ that are invertible.  Moreover, $A^\circ_+$ is a part of $A_+$ and for $a,b\in A_+^\circ$ we have that 
\[
d_{A_+}(a,b) = \|\log (b^{-1/2}ab^{-1/2})\|, 
\]
see for example \cite{ACS}. 

For $x\in A_+^\circ$ we define the linear map $\psi_x\colon \Re(A)\to\Re(A)$ by $\psi_x(a)=x^{-1/2}ax^{-1/2}$. Note that if $a\in A_+$, then  $\psi_x(a)=(x^{-1/2}a^{1/2})(x^{-1/2}a^{1/2})^*$, so that $\psi_x(a)\in A_+$, and hence  
$\psi_x(A_+)\subseteq A_+$. In fact, $\psi_x$ is an invertible linear map that maps $A_+$ onto itself.  It follows from \cite[Corollary 2.1.4]{LNBook} that $\psi_x$ is a Thompson's metric  isometry on $A_+^\circ$.  This isometry will be useful in the sequel.

For $a\in A_+^\circ$, we write 
\[
\lambda_+(a) =\max\{\lambda\colon\lambda\in\sigma(a)\}\mbox{\quad and\quad }
\lambda_-(a) =\min\{\lambda\colon\lambda\in\sigma(a)\}.
\] 
Using this notation, we have for $a,b\in A_+^\circ$ that $a\leq \beta b$ if and only if $b^{-1/2}ab^{-1/2}\leq \beta e$, where $e$ is the unit in $A$. So, 
\[
M(a/b) =\inf\{\beta>0\colon \sigma(\beta e -b^{-1/2}ab^{-1/2})\subseteq [0,\infty)\} = \lambda_+(b^{-1/2}ab^{-1/2}). 
\]
Likewise $\alpha b\leq a$ is equivalent to 
$\sigma(b^{-1/2}ab^{-1/2} -\alpha e)\subseteq [0,\infty)$, so that 
$M(b/a)=m(a/b)^{-1} = \lambda_-(b^{-1/2}ab^{-1/2})^{-1}$.   
So, for $a,b\in A_+^\circ$ we have that  
\[
d_{A_+}(a,b) =\log\left (\max \{\lambda_+(b^{-1/2}ab^{-1/2}),\lambda_-(b^{-1/2}ab^{-1/2})^{-1}\}\right).
\]
We have the following characterization for the unique geodesic in $A^\circ_+$.
\begin{theorem}\label{thm:5.1}
Let $A$ be a unital $C^*$-algebra. If $x$ ad $y$ are linearly independent elements of $ A_+^\circ$, then there exists a unique geodesic segment connecting $x$ and $y$ in $(A_+^\circ,d_{A_+})$ if and only if $\sigma(x^{-1/2}yx^{-1/2})=\{\beta^{-1},\beta\}$ for some $\beta>1$.
\end{theorem}
\begin{proof} 
Note that  there is a unique geodesic segment connecting $x$ and $y$ in $A_+^\circ$ if and only if there is a unique geodesic segment  connecting $\psi_x(x)=e$ and $\psi_x(y)=x^{-1/2}yx^{-1/2}$, as $\psi_x$ is an isometry. Thus, it suffices to show that there is a unique geodesic segment connecting $e$ and $z\in A_+^\circ$ if and only if  
$\sigma(z)=\{\beta^{-1},\beta\}$ for some $\beta>1$ whenever $e$ and $z$ are linearly independent. 

Suppose first that there exists a unique geodesic segment connecting $z$  and $e$ in $A_+^\circ$, where $z$ and $e$ are linearly independent.  It follows from Corollary \ref{cor:3.2} that $\lambda_+(z) = M(z/e) =M(e/z)=\lambda_-(z)^{-1}$. 
This yields the inclusions, 
\[
\{\lambda_+(z)^{-1},\lambda_+(z)\}\subseteq\sigma(z)\subseteq [\lambda_+(z)^{-1},\lambda_+(z)].
\] 
Suppose that there exists $\lambda\in\sigma(z)$ such that $\lambda_+(z)^{-1}<\lambda<\lambda_+(z)$. Let $\delta>0$ be such that $\lambda_+(z)^{-1}<\lambda-\delta<\lambda+\delta<\lambda_+(z)$, then there is a continuous function $f_\delta\colon [\lambda_+(z)^{-1},\lambda_+(z)]\to[0,1]$ with $f_\delta(\lambda)=1$ and $\mathrm{supp}(f_\delta)\subseteq[\lambda-\delta,\lambda+\delta]$. Furthermore, let $g\colon[\lambda_+(z)^{-1},\lambda_+(z)]\to [\lambda_+(z)^{-1},\lambda_+(z)]$ be the identity map. Using the functional calculus $\phi_z\colon C(\sigma(z))\to C^*(z,e)$ we can define $\zeta_\delta=\phi_z(f_\delta)\in C^*(z,e)$, where $C^*(z,e)$ is the $C^*$-algebra generated by $z$ and $e$.  For $\epsilon>0$ we have that $\phi_z(g+\epsilon f_\delta)=z+\epsilon\zeta_\delta$ with $g+\epsilon f_\delta\geq 0$. So, 
\[
\phi_z((g+\epsilon f_\delta)^{1/2})=(z+\epsilon\zeta_\delta)^{1/2}
\] 
and $\sigma((z+\epsilon\zeta_\delta)^{1/2})=\{(g(t)+\epsilon f_\delta(t))^{1/2}:t\in\sigma(z)\}$ by the spectral mapping theorem. So, by choosing $\epsilon>0$ such that 
$(1=\epsilon)(\lambda+\delta)<\lambda_+(z)$, we find that 
\[
\{\lambda_+(z)^{-1/2},\lambda_+(z)^{1/2}\}\subseteq\sigma((z+\epsilon\zeta_\delta)^{1/2})\subseteq[\lambda_+(z)^{-1/2},\lambda_+(z)^{1/2}]
\] 
and $d_{A_+}((z+\epsilon\zeta_\delta)^{1/2},e)=\frac{1}{2}\log(\lambda_+(z))$ for all $\epsilon >0$ sufficiently small. As $\zeta_\delta\in C^*(z,e)$,  
\[
d_{A_+}((z+\epsilon\zeta_\delta)^{1/2},z)=d_{A_+}(z^{-1}(z+\epsilon\zeta_\delta)^{1/2},e).
\]

Now, define  $\xi\in C(\sigma(z))$ by $\xi(t)=t^{-1}(t+\epsilon f_\delta(t))^{1/2}$. 
Again by the functional calculus we have  $\xi(z)=z^{-1}(z+\epsilon\zeta_\delta)^{1/2}$. Moreover,
\[
\xi(t)=\begin{cases} t^{-1/2}& t\in [\lambda_+(z)^{-1},\lambda_+(z)]\setminus[\lambda-\delta,\lambda+\delta],\\ t^{-1/2}\left(1+\frac{\epsilon f_\delta(t)}{t}\right)^{1/2}& t\in [\lambda-\delta,\lambda+\delta],
\end{cases}
\]
and hence $\xi(t)\leq\max\{\lambda_+(z)^{1/2},(\lambda-\delta)^{-1/2}(1+\frac{\epsilon}{\lambda-\delta})^{1/2}\}$ for all $\lambda_+(z)^{-1}\leq t\leq \lambda_+(z)$. For sufficiently small $\epsilon>0$ we can ensure the inequality $\xi(t)\leq \lambda_+(z)^{1/2}$ on $\sigma(z)$. 
As $\xi(t)\geq t^{-1/2}$ and $\xi(\lambda_+(z)^{-1}) = \lambda_+(z)^{1/2} = \xi(\lambda_+(z))^{-1}$,  
\[
\{\lambda_+(z)^{-1/2},\lambda_+(z)^{1/2}\}\subseteq\sigma(\xi(z))\subseteq[\lambda_+(z)^{-1/2},\lambda_+(z)^{1/2}].
\] 
This implies that $d_{A_+}((z+\epsilon\zeta_\delta)^{1/2},z)=d_{A_+}(\xi(z),e)=\frac{1}{2}\log(\lambda_+(z))$ for all $\epsilon >0$ sufficiently small. Note that  $\zeta_\delta\neq 0$, as $f_\delta\neq 0$, which contradicts the fact that there is a unique geodesic segment connecting $z$ and $e$ by Lemma \ref{lem:2.2}. We conclude that $\sigma(z)=\{\lambda_+(z)^{-1},\lambda_+(z)\}$.

Conversely, if $z\in A_+^\circ$ and $e$ are linearly independent and $\sigma(z)=\{\beta^{-1},\beta\}$ for some $\beta>1$, then the function $f:\sigma(z)\to\{0,1\}$ defined by $f(\beta^{-1})=1$ and $f(\beta)=0$ is continuous, and $\beta^{-1}f+\beta(\mathbf{1}-f)$ is the identity function on $\sigma(z)$. So, for $\pi=\phi_z(f)$, it follows that $\beta^{-1}\pi+\beta(e-\pi)=z$ by the functional calculus $\phi_z: C(\sigma(z))\to C^*(z,e)$. Now consider the  2-dimensional closed cone $A_+\cap\mathrm{span}\{e,z\}$, which we can identify with $\mathbb{R}^2_+$. It follows that 
\[
z-m(z/e)e=z-\beta^{-1}e=(\beta-\beta^{-1})(e-\pi)\in\partial\left(A_+\cap\mathrm{span}\{e,z\}\right)
\] 
and
\[
e-m(e/z)z=e-\beta^{-1}z=(1-\beta^{-2})\pi\in\partial\left(A_+\cap\mathrm{span}\{e,z\}\right).\] 
So, for some $\alpha_1,\alpha_2> 0$ we have that 
\[
e'=\alpha_1(1-\beta^{-2})\pi\quad\mbox{and}\quad z'=\alpha_2(\beta-\beta^{-1})(e-\pi)
\] 
are the endpoints in $\partial\left(A_+\cap\mathrm{span}\{e,z\}\right)$ of the straight line segment through $e$ and $z$. Suppose that there is a $v\in\Re(A)$ and an $\epsilon>0$ such that $e'+tv\in \partial A_+$ and $z'+tv\in \partial A_+$ for $|t|<\epsilon$, or equivalently, there is a $\delta>0$ such that $|t|<\delta$ implies $\pi+tv\in \partial A_+$ and $(e-\pi)+tv\in \partial A_+$. By the Gelfand-Naimark theorem, we can view $A$ as a $C^*$-subalgebra of $B(\mathcal{H})$ for some Hilbert space $\mathcal{H}$. So, let $P\colon\mathcal{H}\to\mathcal{H}$ be the projection representing $\pi$ and $V\colon\mathcal{H}\to\mathcal{H}$ be the operator representing $v$. We now have the identities
\[
P=\left(
                         \begin{array}{cc}
                           0 & 0 \\
                           0 & I_2 \\
                         \end{array}
                       \right),
\quad I-P=\left(
            \begin{array}{cc}
              I_1 & 0 \\
              0 & 0 \\
            \end{array}
          \right)\quad\mbox{and}\quad
          V=\left(
              \begin{array}{cc}
                V_1 & V_2 \\
                V_2^* & V_4 \\
              \end{array}
            \right)
\]
relative to $\mathcal{H}=\ker(P)\oplus\mathrm{ran}(P)$. Since $P-tV\ge 0$, it follows that for each $x_1\oplus x_2\in\ker(P)\oplus\mathrm{ran}(P)$ we have
\[
-t\langle V_1(x_1),x_1\rangle-t\langle V_2(x_2),x_1\rangle-t\langle V_2^*(x_1),x_2\rangle-t\langle V_4(x_2),x_2\rangle+\|x_2\|^2\geq 0
\]
 whenever $|t|<\delta$. If we take $0\neq x_1\in\ker(P)$ and $x_2=0$, then $\langle V_1(x_1),x_1\rangle=0$, and hence $V_1=0$, since $V_1$ is self-adjoint. Similarly, the inequality obtained from $(I-P)-tV\ge 0$ for all $|t|<\delta$ implies that $V_4=0$. Now let $0\neq x_2\in\mathrm{ran}(P)$ and $x_1=\alpha V_2(x_2)$, which is an element of $\ker(P)$ for an arbitrary $\alpha\in\mathbb{R}$. Then our findings yield
\[
-2t\alpha\langle V_2(x_2),V_2(x_2)\rangle +\|x_2\|^2=-2t\alpha\|V_2(x_2)\|^2+\|x_2\|^2\geq 0
\] 
whenever $|t|<\delta$. It follows that  $V_2=0$, and therefore also $V_2^*=0$. We conclude from Theorem \ref{thm:4.3} that the geodesic segment connecting $e$ and $z$ is unique. So, if $x,y\in A_+^\circ$ are linearly independent with $\sigma(x^{-1/2}yx^{-1/2})=\{\beta^{-1},\beta\}$ for some $\beta>1$, we have shown that the geodesic segment connecting $e$ and $x^{-1/2}yx^{-1/2}$ is unique, and hence the geodesic segment connecting $x=\psi_{x^{-1}}(e)$ and $\psi_{x^{-1}}(x^{-1/2}yx^{-1/2})=y$ is unique. 
\end{proof}
We have a similar result for Hilbert's metric geodesic segments in $\Sigma_\varphi^\circ$ for some strictly positive functional $\varphi$ on $A$.  Such a functional exists if $A$ is separable. Indeed, in that case,  the {\em state space}, 
$$\mathcal{S}_A=\{\psi\in A^*: \psi\ge 0\ \mbox{and}\ \|\psi\|=1\}$$ 
is $w^*$-metrizable and therefore, since it is also $w^*$-compact by the Banach-Alaoglu's theorem, we must have that $\mathcal{S}_A$ is separable. For a $w^*$-dense sequence $(\varphi_n)_n$ in $\mathcal{S}_A$ we can define the functional 
\[
\varphi=\sum_{n=1}^\infty 2^{-n}\varphi_n.
\] 
Clearly, this defines a positive functional with $\|\varphi\|=1$. Since we have 
\[
\|x\|=\sup\{\psi(x):\psi\in\mathcal{S}_A\}
\] 
for all $x\geq 0$, it follows that $\varphi$ is strictly positive. For more details, see \cite[\S 5.1]{Con} and \cite[\S 5.15]{Con}. 
\begin{theorem}\label{thm:5.2} Let $A$ be a unital $C^*$-algebra with a strictly positive functional $\varphi$. For distinct $x,y\in\Sigma_\varphi^\circ$, there exists a unique Hilbert's metric geodesic segment connecting $x$ and $y$ in $\Sigma_\varphi^\circ$ if and only if $\sigma(x^{-1/2}yx^{-1/2})=\{\alpha,\beta\}$ for some $\beta>\alpha>0$.
\end{theorem}
\begin{proof} Suppose the straight line segment connecting $x$ and $y$ is the unique Hilbert's metric geodesic segment  in $\Sigma_\varphi^\circ$. For $\lambda = M(x/y)^{1/2}M(y/x)^{-1/2}$
we have 
\[
M(x/\lambda y)=M(x/y)^{1/2}M(y/x)^{1/2}=M(\lambda y/x). 
\]
So, there exists a unique type I geodesic segment connecting $x$ and $\lambda y$ in $A_+^\circ\cap\mathrm{span}\{x,y\}$. By Lemma \ref{lem:4.4} this geodesic segment is unique in $(A_+^\circ,d_{A_+})$. Now  Theorem  \ref{thm:5.1} implies that 
\[
\sigma(x^{-1/2}(\lambda y)x^{-1/2})=\lambda^{-1}\sigma(x^{-1/2}yx^{-1/2})=
\{\beta^{-1},\beta\}
\] 
for some $\beta>1$, or equivalently, $\sigma(x^{-1/2}yx^{-1/2})=
\{\lambda\beta^{-1},\lambda\beta\}$.

Conversely, if $\sigma(x^{-1/2}yx^{-1/2})=\{\alpha,\beta\}$ for some $\beta>\alpha>0$. Then for $\mu=\sqrt{\alpha\beta}$ and $\xi=\sqrt{\beta/\alpha}$, we can write 
\[
\sigma(x^{-1/2}(\mu^{-1}y)x^{-1/2})=\{\xi^{-1},\xi\}. 
\]
So, Theorem \ref{thm:5.1} implies that the Thompson's metric geodesic segment  connecting $x$ and $\mu^{-1}y$ in $A_+^\circ$ is unique. So, by Lemma \ref{lem:4.4} 
the straight line segment connecting $x$ and $y$ in $\Sigma_\varphi^\circ$ is the 
unique Hilbert's metric geodesic segment.
\end{proof}

\section{Unique  geodesics in symmetric cones}

Recall that the interior $K^\circ$ of a closed cone $K$ in a finite-dimensional inner-product space $(V,\langle\cdot,\cdot\rangle)$ is called a {\em symmetric cone} if $K$  the {\em dual cone}, $K^*=\{y\in V\colon \langle y,x\rangle\geq 0\mbox{ for all }x\in K\}$ satisfies  $K^*=K$, and $\mathrm{Aut}(K)=\{A\in\mathrm{GL}(V)\colon A(K)=K\}$ acts transitively on $K^\circ$. A prime example is the cone of positive definite Hermitian matrices. 
In this section we prove a characterization of the unique Thompson metric geodesics in symmetric cones $K^\circ$, which is similar to the one given in Theorem \ref{thm:5.1}.  

It is well known that the symmetric cones in finite dimensions are precisely the interiors of  the cones of squares of Euclidean Jordan algebras. This fundamental result is due to Koecher \cite{Koe} and Vinberg \cite{Vin}. A detailed exposition of the theory of symmetric cones can be found in the book by Faraut and Kor\'anyi \cite{FK}.  We will follow their notation and terminology. Recall that a {\em Euclidean Jordan algebra} is a finite-dimensional real inner-product space $(V,\langle\cdot,\cdot\rangle)$ equipped with a bilinear product $(x,y)\mapsto x\bullet y$ from $V\times V$ into $V$ such that for each $x,y\in V$:
\begin{enumerate}
\item[(1)] $x\bullet y=y\bullet x$, 
\item[(2)] $x\bullet (x^2\bullet y)=x^2\bullet (x\bullet y)$, and 
\item[(3)] for each $x\in V$, the linear map $L(x)\colon V\to V$ given by $L(x)y =x\bullet y$ 
satisfies 
\[
\langle L(x)y, z\rangle = \langle y, L(x)z\rangle\mbox{\quad for all }y,z\in V. 
\]
\end{enumerate}
In general a Euclidean Jordan algebra is not  associative, but it is  commutative.  
We denote the {\em unit} in a Euclidean Jordan algebra by $e$. 
An element $c\in V$ is called an {\em idempotent}  if $c^2=c$.  
 A set $\{c_1,\ldots,c_k\}$ is called a {\em complete system of orthogonal idempotents} if 
 \begin{enumerate}[(1)]
\item  $c_i^2= c_i$  for all $i$, 
\item $c_i\bullet c_j =0$ for all $i\neq j$, and 
\item $c_1+\cdots+ c_k =e$. 
\end{enumerate}

The spectral theorem \cite[Theorem III.1.1]{FK} says that for each $x\in V$ there exist unique real numbers $\lambda_1, \ldots,\lambda_k$, all distinct, and a complete system of orthogonal idempotents $c_1,\ldots,c_k$ such that 
$x = \lambda_1 c_1+\cdots +\lambda_k c_k$. 
The numbers $\lambda_i$ are called the {\em eigenvalues} of $x$. The {\em spectrum} \index{spectrum} of $x$ is denoted  by 
$\sigma(x) =\{\lambda \colon \lambda \mbox{ eigenvalue of } x\}$,  
and we write 
\[\lambda_+(x)=\max\{\lambda\colon\lambda\in\sigma(x)\} \mbox{\quad and\quad } 
\lambda_-(x)=\min\{\lambda\colon\lambda\in\sigma(x)\}.
\]

It is known, see for example \cite[Theorem III.2.2]{FK}, that $x\in K^\circ$ if and only if $\sigma(x)\subseteq (0,\infty)$, which is equivalent to $L(x)$ being positive definite. So, one can use the spectral decomposition,  $x =\lambda_1 c_1+\cdots +\lambda_kc_k$, of $x\in K^\circ$, to define a spectral calculus, e.g., 
\[
x^{-1/2} =\lambda_1^{-1/2}c_1+\cdots+\lambda_k^{-1/2}c_k.
\]
For $x\in V$   the linear mapping, $P(x) = 2L(x)^2 -L(x^2)$, 
is called the {\em quadratic representation} of $x$. Note that $P(x^{-1/2})x=e$ for all $x\in K^\circ$. It is known that $P(x^{-1}) = P(x)^{-1}$ for all $x\in K^\circ$ and  $P(x)\in\mathrm{Aut}(K)$ whenever $x\in K^\circ$, see \cite[Proposition III.2.2]{FK}.
So, $P(x)$ is an isometry of $(K^\circ,d_K)$ if $x\in K^\circ$ by  \cite[Corollary 2.1.4]{LNBook}. 
For $x,y\in K^\circ$ we write 
\[
\lambda_+(x,y) =\lambda_+(P(y^{-1/2}) x)\mbox{\quad and\quad }
\lambda_-(x,y) =\lambda_-(P(y^{-1/2}) x).
\]
Note that for $x,y\in K^\circ$,  $x\leq \beta y$ if and only if $0\leq \beta e -P(y^{-1/2})x$, and hence 
\[
M(x/y) = \lambda_+(x,y).\]
Similarly, $\alpha y\leq x$ is equivalent with $0\leq P(y^{-1/2})x-\alpha e$, and hence 
\[
M(y/x)^{-1} = m(x/y) = \lambda_-(x,y).
\]
So, for $x,y\in K^\circ$ the Thompson metric distance is given by  
\[
d_K(x,y) = \log\left( \max\{\lambda_+(x,y), \lambda_-(x,y)^{-1}\} \right).
\]

The following lemma is Exercise 3.3 in \cite{FK}. For the sake of completeness, we will give a proof.
\begin{lemma}\label{lem:6.1} Let $V$ be a Euclidean Jordan algebra with symmetric cone $K^\circ$. For $x,y\in K$ we have $\langle x,y\rangle=0$ if and only if $x\bullet y=0$.
\end{lemma}
\begin{proof}Without loss of generality, we may assume that $x,y\in K\setminus\{0\}$. Suppose that $\langle x,y\rangle=0$. Write $y=v^2$ for some $v\in V$. It follows that 
\[
\langle x,v\bullet v\rangle =\langle L(v)x,v\rangle =\langle L(x)v,v\rangle =0.
\] 
Since $L(x):V\to V$ is a self-adjoint positive semi-definite linear map, we know that $L(x)^{1/2}$ is well defined, which yields 
\[
\|L(x)^{1/2}v\|^2=\langle L(x)^{1/2}v,L(x)^{1/2}v\rangle =0.
\] 
It follows that $L(x)v=L(x)^{\frac{1}{2}}(L(x)^{\frac{1}{2}}v)=0$. However, $L(y)$ and $L(v)$ commute, so that 
\begin{eqnarray*}
\langle L(x)y,y\rangle & = &\langle L(x)y,L(v)v\rangle \\
                  & = &\langle L(v) (L(y)x),v\rangle \\
                  & = & \langle L(y)(L(v)x),v\rangle \\
                  & = & \langle L(y)(L(x)v),v\rangle =0.
\end{eqnarray*}
Using the same argument as above, we deduce that $L(x)y=x\bullet y=0$.

Obviously, if $x\bullet y=0$, then $\langle x,y\rangle =\langle e,L(x)y\rangle =\langle e,x\bullet y\rangle =0$.
\end{proof}
We can now prove the analogue of Theorem \ref{thm:5.1} for symmetric cones. 
\begin{theorem}\label{thm:4.2} Let $V$ be a Euclidean Jordan algebra with symmetric cone $K^\circ$. If $x,y\in K^\circ$ are linearly independent, then there exists a unique geodesic segment connecting $x$ and $y$ in $(K^\circ,d_K)$ if and only if  $\sigma(P(y^{-\frac{1}{2}})x)=\{\beta^{-1},\beta\}$ for some $\beta>1$.
\end{theorem}
\begin{proof} Suppose that there exists a unique geodesic segment  connecting two linearly independent elements $x,y\in K^\circ$. As $y^{-1/2}\in K^\circ$, $P(y^{-1/2})\in\mathrm{Aut}(K)$ (see \cite[Theorem III.2.2]{FK}), and hence $P(y^{-1/2})$ is an isometry of $(K^\circ,d_K)$ by  \cite[Corollary 2.1.4]{LNBook}.
Thus, there exists a unique geodesic segment connecting $x$ and $y$ in $K^\circ$ if and only if there is a unique geodesic connecting  $P(y^{-1/2})y=e$ and $P(y^{-1/2})x$. So, it suffices to show that if there exists a unique geodesic segment connecting $e$ and $z\in K^\circ$ with $e$ and $z$  linearly independent, 
then $\sigma(z)=\{\beta,1/\beta\}$ for some $\beta>1$. 

Let $z=\lambda_1 c_1+\cdots+\lambda_k c_k$ be the spectral decomposition of $z$. We have that 
\[
\lambda_+(z) =M(z/e)\mbox{\quad and\quad }\lambda_-(z) = m(z/e)=M(e/z)^{-1}.
\] 
As there exists a unique geodesic segment connecting $z$ and $e$, we have that 
$M(z/e) =M(e/z)$ by Corollary \ref{cor:3.2}. So, if we write $r=\lambda_+(z)$, then  
\[
\{1/r,r\}\subseteq \sigma(z)\subseteq [1/r,r].
\]
Note that $d_K(e,z)=\log M(z/e)>0$, and hence $r>1$.  

Suppose there exists $\lambda_i\in\sigma(z)$ with $1/r<\lambda_i<r$. For $\epsilon >0$ we define $z_\epsilon =(z+\epsilon c_i)^{1/2}$. If $\lambda_i+\epsilon <r$,  we can use the spectral decomposition of $z$ to find that $d_K(e,z_\epsilon)=\frac{1}{2}\log r$. Also, note that $d_K(z_\epsilon,z)=d_K(P(z^{-1/2})z_\epsilon,e)$ and 
\[
P(z^{-1/2})z_\epsilon=\lambda_1^{-1/2}c_1+\cdots+\frac{(\lambda_i+\epsilon )^{1/2}}{\lambda_i}c_i+\cdots+\lambda_k^{-1/2}c_k.\]
As $0<1/r<\lambda_i<r$, we have that $r^{-1/2}<\lambda_i^{-1/2}<(\lambda_i+\epsilon)^{1/2}/\lambda_i$ and $r\lambda_i^2 -\lambda_i>0$.
So, for $0<\epsilon<r\lambda_i^2-\lambda_i$,  
\[
r^{-1/2}<\frac{(\lambda_i+\epsilon )^{1/2}}{\lambda_i}<r^{1/2}. 
\]
Thus, $d_K(z_\epsilon,z)=\textstyle{\frac{1}{2}}\log r$ and $d_K(e,z_\epsilon)=\frac{1}{2}\log r$ for all $0<\epsilon<\min\{r-\lambda_i, r\lambda_i^2 -\lambda_i\}$.   This  is impossible by Lemma \ref{lem:2.2} and therefore  $\sigma(z)=\{1/r,r\}$.

Conversely, suppose that $z$ and $e$ in $K^\circ$ are linearly independent, and $\sigma(z)=\{1/\beta,\beta\}$ for some $\beta>1$. Then we have the spectral decomposition $z=\beta^{-1}c_1+\beta c_2$. Note that, as $\sigma(z)=\{\beta,1/\beta\}$, $M(z/e)=M(e/z)=\beta >1$. So, the straight line through $e$ and $z$ intersects  $\partial K$ in 2 points $e'$ and $x'$ by Lemma \ref{lem:4.2}. In fact, $(\beta -1/\beta)c_1 =\beta e -z\in\partial K$ and $(\beta^2-1)c_2 =\beta z-e\in\partial K$, so that  $e'=\lambda_1 c_1$ and $z'=\lambda c_2$ for some $\lambda_1,\lambda_2>0$. Suppose there exist $\epsilon >0$ and $v\in V$ such that $c_1+tv\in \partial K$ and $c_2+tv\in \partial K$ whenever $|t|<\epsilon$.  So, for $|t|<\epsilon$ the operator $L(c_1+tv)$ is positive semi-definite, which yields 
\[
0\leq\langle L(c_1+tv)c_2,c_2\rangle=t\langle v\bullet c_2,c_2\rangle =t\langle v,c_2\rangle , 
\] 
and hence $\langle v,c_2\rangle =0$. It follows from Lemma \ref{lem:6.1} that $v\bullet c_2=0$. In a similar way we find that $v\bullet c_1=0$, so $v=v\bullet e=v\bullet(c_1+c_2)=0$ and we conclude from Theorem \ref{thm:4.3} that the geodesic segment connecting $e$ and $z$ in $(K^\circ,d_K)$ is unique. We have shown that if $x,y\in K^\circ$ are linearly independent and $\sigma(P(y^{-1/2})x)=\{\beta,1/\beta\}$ for some $\beta>1$, that there exists a unique geodesic segment connecting $e$ and $P(y^{-1/2})x$ in $(K^\circ,d_K)$. Equivalently, there is a unique geodesic segment connecting $y=P(y^{1/2})e$ and $x=P(y^{1/2})(P(y^{-1/2})x)$. 
\end{proof} 

As for the characterization of unique geodesic segments in $(\Sigma_\varphi^\circ,\delta_K)$ for some strictly positive functional $\varphi$ on $V$, we also have an analogue of Theorem \ref{thm:5.2}. The proof is completely analogous and is left to the reader.  
\begin{theorem}\label{thm:6.3} Let $V$ be a Euclidean Jordan algebra with symmetric cone $K^\circ$,  $\varphi\in K^\circ$, and $\Sigma_\varphi^\circ=\{x\in K^\circ\colon \langle \varphi,x\rangle=1\}$. For distict  $x,y\in\Sigma_\varphi^\circ$  there exists a unique geodesic segment connecting $x$ and $y$ in $(\Sigma_\varphi^\circ,\delta_K)$ if and only if $\sigma(P(y^{-1/2})x)=\{\alpha,\beta\}$ for some $\beta>\alpha>0$.
\end{theorem}

\section{Quasi-isometric embeddings into normed spaces}
In this section we will study isometric and quasi-isometric embeddings of Thompson's metric spaces $(C^\circ,d_C)$, where $C^\circ$ is the interior of a finite-dimensional closed cone, into finite-dimensional normed spaces. Recall that a map $f$ from a metric space $(X,d_X)$ into a metric space $(Y,d_Y)$ is a called a {\em quasi-isometric embedding} if there exist constants $\alpha\geq 1$ and $\beta\geq 0$ such that 
\[
\frac{1}{\alpha} d_X(x,y) -\beta \leq d_Y(f(x),f(y))\leq \alpha d_X(x,y)+\beta\mbox{\quad for all }x,y\in X.
\]

It is known that if $C$ is a polyhedral cone with $N$ facets, then $(C^\circ,d_C)$ can be isometrically embedded into $(\mathbb{R}^N,\|\cdot\|_\infty)$, see \cite[Lemma 2.2.2]{LNBook}. We will show that polyhedral cones are the only ones that allow a quasi-isometric embedding into a finite-dimensional normed space. 

A similar result  exists for Hilbert's metric spaces and  was proved  by Colbois and Verovic in \cite{CV}. The idea of their proof can be traced back to \cite{FK} and  relies on properties of the Gromov product in Hilbert's metric spaces proved in \cite[Theorem 5.2]{KN}. It turns out that the usual Gromov product does not have the right behavior in Thompson's metric spaces. The following generalized Gromov product, however, will be useful. 
\begin{definition}\label{gen_gromov}
Let $(X,d_X)$ be a metric space. For $p\in X$ and $\eta>0$ the {\em generalized Gromov product} on $X\times X$ is given by 
\[
(x|y)_{p,\eta} = \frac{1}{2}\left ( d_X(x,p)+d_X(y,p) -\eta d_X(x,y)\right).
\]
\end{definition}
Note that for $\eta=1$ we recover the usual Gromov product. It turns out that for  Thompson's metric the generalized Gromov product where $\eta=2$ is  relevant.

The following lemma is a slight generalization of \cite[Proposition 2.1]{CV}.
\begin{lemma}\label{lem:7.1} Let $(X,d_X)$ be  a metric space that  can be quasi-isometrically embedded into a finite-dimensional normed space $(V,\|\cdot\|)$. If there exist $p\in X$, a constant $\eta>0$, and sequences  $(x_k^i)_k$ for $i=1,\ldots,m$  such that $d_X(x_k^i,p)=k$ for all $i=1,\ldots,m$ and $k\geq 1$, and 
\[
\limsup_{k\to\infty}\,(x^i_k| x^j_k)_{p,\eta}\leq C_{ij}<\infty
\]
 for all $i\neq j$, then there exist $v^1,\ldots, v^m\in V$ satisfying 
 \[
 \|v^i-v^j\|\ge \frac{2}{\alpha\eta}\mbox{\quad for all $i\neq j$}
 \] 
 and $1/\alpha\le\|v^i\|\le\alpha$ for all $i$, where $\alpha\geq 1$ is the constant from the quasi-isometry.
\end{lemma}
\begin{proof}
Let $f\colon X\to V$ be a quasi-isometric embedding. We may as well assume that  $f(0)=0$, as the map $g(x)=f(x)-f(p )$ is also a quasi-isometric embedding. Now for $ i\neq j$ there exists a number $N\geq 1$ and a constant $R<\infty $ such that 
\[
d_X(x^i_k,x^j_k)\geq\frac{d_X(x^i_k,p)+d_X(x^j_k,p)-R}{\eta}=\frac{2k-R}{\eta}
\] 
whenever $k\geq N$. Define the vectors $u^i_k=\frac{1}{k}f(x^i_k)\in V$ for all $k\geq 1$ and $i=1,\ldots,m$. It follows that 
\[
\|u_k^i-u_k^j\|=\frac{1}{k}\| f(x^i_k)-f(x^j_k)\|\ge \frac{1}{\alpha k}d_X(x^i_k, x^j_k)-
\frac{\beta}{k}\geq\frac{2}{\alpha \eta}- \frac{1}{k}\left( \frac{R}{\alpha\eta}+\beta\right)
\] 
whenever $k\geq 1$ and $i\neq j$. Also, we have that 
\[
\frac{1}{\alpha}-\frac{\beta}{k}\le\|u_k^i\|=\frac{1}{k}\|f(x^i_k)-f(p)\|\leq\alpha+\frac{\beta}{k}\]
for all $k\geq 1$ and $1\leq i\leq m$. Since $V$ is finite-dimensional, there are convergent subsequences $(u^i_{k_j})_j$ with limits $v^i$ for $i=1,\ldots, m$. 
The vectors $v^i$ have the desired properties. 
\end{proof}

Lemma \ref{lem:7.1} has the following consequence. 
\begin{corollary}\label{cor:7.2} 
If $(X,d_X)$ is a metric space and there exist  $p\in X$, a constant $\eta>0$, and  sequences  $(x_k^i)_k$  in $X$ for $i=1,2,\ldots$ such that $d(x_k^i,p)=k$ for all $i\geq 1$ and $k\geq 1$, and 
\[
\limsup_{k\to\infty}\,(x^i_k| x^j_k)_{p,\eta}\leq C_{ij}<\infty
\]
 for all $i\neq j$, then $(X,d_X)$ cannot be quasi-isometrically embedded into a finite-dimensional normed space. 
\end{corollary}
\begin{proof}
Suppose that $(X,d_X)$ can be quasi-isometrically embedded into a finite-dimensional normed space $(V,\|\cdot\|)$ using a quasi-isometry with constants $\alpha\geq 1$ and $\beta \geq 0$. As the set $S=\{x\in V\colon 1/\alpha\leq  \|x\|\leq \alpha\}$ is compact,  the maximum number of points in $S$ whose pairwise distance is at least $2/(\alpha\eta)$ is bounded by a constant $M_{\alpha,\eta}<\infty$. This contradicts  Lemma \ref{lem:7.1}.
\end{proof}

We will see that if $C$ is not a polyhedral cone, then we can find infinitely many sequences $(x^i_k)_k$ in $(C^\circ,d_C)$ satisfying the conditions in Corollary 
\ref{cor:7.2} with $\eta=2$.  We will need the following auxiliary lemma. 
\begin{lemma}\label{lem:7.4} 
Let $C$ be a closed  cone with nonempty interior  in a finite-dimensional normed space $(V,\|\cdot\|)$. If $S\subseteq C ^\circ$ is a norm compact subset and $(x_k)_k$ is a sequence in $C^\circ$ such that $x_k\to x\in\partial C$, then there exists $N\geq 1$ such that $d_C(x_k,s)=\log M(s/x_k)$ for all $s\in S$ and all $k\ge N$.
\end{lemma}
\begin{proof}
Let $u\in C^\circ$ and $\Sigma_u^*=\{\varphi \in C^*\colon \varphi(u)=1\}$. As $C^*$ is a closed cone with nonempty interior in $V^*$, we know from \cite[Lemma 1.2.4]{LNBook} that $\Sigma_u^*$ is a compact set of $V^*$, and hence there exists a constant $M_1>0$ such that $\|\varphi\|\leq M_1$ for all $\varphi\in\Sigma^*_u$. 
Define functions $f\colon C^\circ\to\mathbb{R}$ and $g\colon C^\circ\to\mathbb{R}$  by 
\[f(x)=\min_{\varphi\in\Sigma_u^*}\varphi(x)\quad\mbox{and}\quad g(x)=\max_{\varphi\in\Sigma_u^*}\varphi(x)\mbox{\quad 
for $x\in C^\circ$.}
\] 

The topology on $C^\circ$ generated by $d_C$ is the same as the norm topology by \cite[Corollary 2.5.6]{LNBook}.  
Note that there exists a constant $M_2>0$ such that $\|s\|\leq M_2$ for all $s\in S$, as $S$ is compact. Thus, $g(s)\leq \max_{\varphi\in\Sigma_u^*} \|\varphi\|\|s\|\leq M_1M_2$ for all $s\in S$. 
Also, if $|f(x)-f(y)| = f(x)-f(y)$ and $f(y) = \psi(y)$ with $\psi\in\Sigma_u^*$, then 
$|f(x)-f(y)| = f(x)-f(y)\leq \psi(x)-\psi(y)\leq \|\psi\|\|x-y\|\leq M_1\|x-y\|$.  Thus, $f$ is a continuous function, and hence $\delta=\min_{s\in S}f(s)>0$. 

For $x_k\in C^\circ$ we know, by \cite[Lemma 1.2.1]{LNBook}, that 
\[
\sup_{s\in S}M(x_k/s)=\sup_{s\in S}\left(\max_{\varphi\in\Sigma_u^*}\frac{\varphi(x_k)}{\varphi(s)}\right)\leq\frac{g(x_k)}{\delta}\leq\frac{M_1M_2}{\delta}.
\]
As $x\in\partial C$, there exists $\rho\in\Sigma_u^*$ such that $\rho(x)=0$. This implies that there exists $N\geq 1$ such that 
$\delta/\rho(x_k)>M_1M_2/\delta$ for all $k\geq N$, and hence 
\[
M(s/x_k)= \max_{\varphi\in\Sigma_u^*}\frac{\varphi(s)}{\varphi(x_k)} \geq \max_{\varphi\in\Sigma_u^*}\ \frac{ f(s)}{\phi(x_k)}\geq \frac{ \delta}{\rho(x_k)}> M(x_k/s)
\] for all $s\in S$ and $k\geq N$. Thus,  $d_C(x_k,s)=\log M(s/x_k)$ for all $s\in S$ whenever $k\geq N$.
\end{proof}

The following result is the analogue of \cite[Theorem 5.2]{KN} for Thompson's metric. 
\begin{proposition}\label{prop:7.5} Let $C$ be a closed cone with nonempty interior  in a finite-dimensional normed space $(V,\|\cdot\|)$, $\varphi\in C^*$ strictly positive, and $\Sigma_\varphi^\circ=\{x\in C^\circ\colon\varphi(x)=1\}$. Suppose that  $(x_n)_n$ and $(y_n)_n$ are convergent sequences in $\Sigma_\varphi^\circ$ with $x_n\to x\in\partial C$ and $y_n\to y\in\partial C$. If $tx+(1-t)y\in C^\circ$ for all $0<t<1$ and $p\in C^\circ$, then 
\[
\limsup_{k\to\infty}\,(x_k|y_k)_{p,2}<\infty.
\]
\end{proposition}
\begin{proof} 
For $k\geq 1$ let $z_k=\frac{1}{2}(x_k+y_k)$ and $z=\frac{1}{2}(x+y)\in C^\circ$. Note that $z_k\to z$ as $k\to\infty$. Let $\epsilon>0$ be such that the closed norm ball, $B_\epsilon$, with radius $\epsilon$ and center $z$ is contained in $C^\circ$. There exists a number $M\geq 1$ such that $z_k\in B_\epsilon$ for all  $k\geq M$. By Lemma \ref{lem:7.4} there also exists a number $N\geq M$ such that $d_C(x_k,s)=\log M(s/x_k)$ and $d_C(y_k,s)=\log M(s/y_k)$ for all $k\geq N$ and all $s\in B_\epsilon$. 
Let $x'_k$ and $y'_k$ be the points of intersection of the straight line through $x_k$ and $y_k$ with $\partial C$ such that $x_k$ is between $x'_k$ and $y_k$, and $y_k$ is between $y'_k$ and $x_k$. 
As was shown in the proof of Theorem \ref{thm:4.3}, we have that 
\[
\log M(x_k/y_k)=\log\frac{\|x_k-y'_k\|}{\|y_k-y'_k\|}\quad\mbox{and}\quad\log M(y_k/x_k)=\log\frac{\|y_k-x'_k\|}{\|x_k-x'_k\|}.
\]
It follows that
\begin{eqnarray*} 
d_C(x_k,y_k) \geq   \log M(x_k/y_k)
&  = &  \log \frac{\|x_k-y'_k\|}{\|y_k-y'_k\|}\\
  & \geq &\log\frac{\|z_k-y'_k\|}{\|y_k-y'_k\|}
  = \log M(z_k/y_k)
 =  d_C(z_k,y_k).
\end{eqnarray*} 
Similarly, $d_C(x_k,y_k)\ge d_C(x_k,z_k)$. As the norm topology coincides with the Thompson's metric topology on $C^\circ$, these inequalities, finally imply that 
\begin{eqnarray*}\limsup_{k\to\infty}2(x_k|y_k)_{p,2}&\leq & \limsup_{k\to\infty}d_C(x_k,p)-d_C(x_k,z_k)+d_C(y_k,p)-d_C(y_k,z_k)\\
&\leq& \limsup_{k\to\infty}\,2d_C(z_k,p)\\
 & \leq &2d_C(z,p).
\end{eqnarray*}
\end{proof}

Recall that if $C$ is a closed polyhedral cone with nonempty interior in a finite-dimensional vector space $V$, the dual cone is also a polyhedral cone. Indeed, as $C^{**} = C$ whenever $C$ is a closed  finite-dimensional cone with nonempty interior, we know that if $C$ is  non-polyhedral, then $C^*$ is also non-polyhedral, see \cite[Corollary 19.2.2]{Rock}. 
The following notions play a role in the  proof of the next result.   
A face $F$ of a closed cone $C$ is called an {\em extreme ray}  if $\dim F=1$. An extreme ray $F$ of  $C$ is said to be an {\em exposed ray} if there exists $\varphi\in C^*$ such that 
$F=\{x\in C\colon \varphi(x)=0\}$. The cone version of Strazewicz's theorem \cite[p.167]{Rock} says that in a finite-dimensional  closed cone $C$ the exposed rays are dense in the extreme rays, i.e.,  the norm closure of $\{x\in C\colon x\mbox{ on an exposed ray of } C\}$ coincides with the norm closure of $\{x\in C\colon x\mbox{ on an extreme ray of } C\}$.

\begin{theorem}\label{thm:7.6} If $C$ is a closed finite-dimensional cone with nonempty interior,  then $(C^\circ,d_C)$ can be quasi-isometrically embedded into a finite-dimensional normed space if and only if $C$ is a polyhedral cone. 
\end{theorem}
\begin{proof}
It  is known that if $C$ is a  closed polyhedral cone with nonempty interior, then $(C^\circ,d_C)$ can be isometrically embedded into  $(\mathbb{R}^m,\|\cdot\|_\infty)$, where $m$ is the number of facets of $C$, see \cite[Lemma 2.2.2]{LNBook}. 

To prove the converse, let $C$ be a closed non-polyhedral cone with nonempty interior in a finite-dimensional  vector space $V$. 
As $C$ is a closed non-polyhedral cone with nonempty interior,  $C^*$ is a also a closed non-polyhedral cone with nonempty interior in $V^*$. So, $C^*$ has infinitely many extreme rays. By the cone version of Strazewicz's theorem \cite[p.167]{Rock}, $C^*$ has infinitely many exposed rays. Let $\xi\in C^*$ be a strictly positive functional and let  $\Sigma_\xi^\circ=\{x\in C^\circ \colon \xi(x)=1\}$.

For each integer $i\geq 1$, select distinct $\psi_i\in\partial C^*$ such that $F_i=\{\mu \psi_i\colon \mu\geq 0\}$ is an exposed ray of $C^*$. This means that there exists $w_i\in V^{**}=V$  with $w_i\in\partial C$ with $\xi(w_i)=1$ such that 
$F_i=\{\varphi\in C^*\colon \varphi(w_i)=0\}$. 
So, $\varphi(w_i)>0$ whenever $\varphi\in C^*$ and $\varphi\neq \mu \psi_i$ for all $\mu\geq 0$. 

Clearly, if $i\neq j$ and $0<\lambda<1$, then  $\varphi(\lambda w_i+(1-\lambda)w_j)>0 $ for all $\varphi\in C^*\setminus\{0\}$. This implies that $\lambda w_i+(1-\lambda)w_j\in C^\circ$ for all $i\neq j$ and $0<\lambda<1$, see \cite[Theorem 11.2]{Rock}.

Take $p\in \Sigma_\xi^\circ$ fixed. For $i\geq 1$  and  $0<t<1$ let 
\[
\gamma_i(t) = tw_i + (1-t)p.
\] 
As the norm topology of $d_C$ coincides with the topology on $C^\circ$, the maps, $t\mapsto d_C(\gamma_i(t),p)$, are continuous on $(0,1)$ for all $i\geq 1$. Moreover, $d_C(\gamma_i(t),p)\to\infty$ as $t\to 1$. Thus, for each $i\geq 1$, there exists a strictly increasing sequence $(t^i_{k})_k$ in $(0,1)$  with $t^i_{k}\to 1$ as $k\to\infty$ such that $d_C(\gamma_i(t^i_{k}),p)=k$ for all $k\geq 1$. If we let $x^i_k= \gamma_i(t^i_{k})$ in $\Sigma_\xi^\circ$, the sequences $(x^i_k)_k$ in $(C^\circ, d_C)$ satisfy the conditions of Corollary \ref{cor:7.2} by Proposition \ref{prop:7.5}, and hence $(C^\circ,d_C)$ cannot be quasi-isometrically embedded into a finite-dimensional normed space. 

\end{proof}

We can use Theorems \ref{thm:4.3} and \ref{thm:7.6} to prove the following characterization of simplicial cones, which is the analogue of \cite[Theorem 2]{FK} 
for Thompson's metric spaces. 

\begin{theorem}\label{thm:7.7 }If  $C$ is a closed finite-dimensional cone with nonempty interior, then 
$(C^\circ,d_C)$ is isometric to a finite-dimensional normed space if and only if $C$ is a simplicial cone.
\end{theorem}
\begin{proof}
Suppose that that $C$ is not simplicial and that $f$ is an isometry of $(C^\circ,d_C)$ onto a finite-dimensional normed space $(V,\|\cdot\|)$. Let $\varphi\in C^*$ be strictly positive and  $\Sigma_\varphi^\circ=\{x\in C^\circ\colon \varphi(x)=1\}$. By Theorem \ref{thm:7.6} we have that $C$ is a polyhedral cone, so $\Sigma_\varphi^\circ$ is the interior of a polytope. Since $C$ is not simplicial, it follows that $\Sigma_\varphi^\circ $ is not the interior of  an $(n-1)$-simplex, where $n=\dim(V)$. This implies that there exist vertices $v_1$ and $v_2$ in $\partial \Sigma_\varphi^\circ$ and $u\in\partial \Sigma_\varphi^\circ$ such that $t v_1+(1-t)u\in\Sigma_\varphi^\circ$ and $tv_2+(1-t)u\in\Sigma_\varphi^\circ$ for all $0<t<1$ and $u$ is not a vertex. The situation is depicted in the Figure \ref{fig:isom}.

\begin{figure}[h]
\begin{center}
\thicklines
\begin{picture}(100,70)(0,0)
\put(0,0){\line(1,0){100}} 
\put(50,0){\circle*{3.0}}\put(47,10){$w$}
\put(0,50){\circle*{3.0}}\put(0,60){$v_1$}
\put(100,50){\circle*{3.0}}\put(100,60){$v_2$}
\put(50,0){\line(1,1){50}} 
\put(50,0){\line(-1,1){50}} 
\end{picture}
\caption{Vertices \label{fig:isom}}
\end{center}
\end{figure}
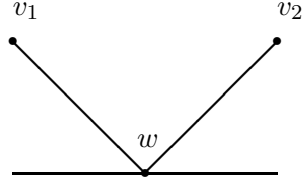

Let $\gamma_1(t)=e^t u+e^{-t}v_1$ and $\gamma_2(t)=e^t u+e^{-t}v_2$, for $t\in\mathbb{R}$, be type I geodesics in $(C^\circ,d_C)$, see Lemma \ref{lem:3.2}.  As $v_1$ and $v_2$ are vertices of $\Sigma_\varphi$, it follows from Theorem \ref{thm:4.3} that both $\gamma_1$ and $\gamma_2$ are unique geodesic lines in $(C^\circ,d_C)$. This implies that the images of $\gamma_1$ and $\gamma_2$ under the isometry $f$, which we will denote by $\ell_1$ and $\ell_2$, respectively, are straight lines in $V$, since unique geodesic lines  are mapped to unique geodesic lines by $f$. 

Now, fix $\xi\in C^\circ$ and let  $\Sigma_\xi^*=\{\varphi\in C^*:\varphi(\xi)=1\}$.  
For $x,y\in C^\circ$ we have that 
\[
M(x/y)=\sup_{\varphi\in \Sigma^*_\xi} \frac{\varphi(x)}{\varphi(y)}, 
\]
see \cite[p.34]{LNBook}. Note that for $\varphi\in\Sigma_\xi^*$ and $t\in\mathbb{R}$ we have
\[
\frac{\varphi(\gamma_1(t))}{\varphi(\gamma_2(t))}=\frac{e^t\varphi(u)+e^{-t}\varphi(v_1)}{e^t\varphi(u)+e^{-t}\varphi(v_2)}=\frac{\varphi(u)+e^{-2t}\varphi(v_1)}{\varphi(u)+e^{-2t}\varphi(v_2)}.
\] 
So, if $\varphi(u)=0$, then neither $\varphi(v_1)$ nor $\varphi(v_2)$ can be $0$, and we find that $$\frac{\varphi(\gamma_1(t))}{\varphi(\gamma_2(t))}=\frac{\varphi(v_1)}{\varphi(v_2)}<\infty$$ for all $t\in\mathbb{R}$. On the other hand, if $\varphi(u)\neq 0$, then 
\[
\frac{\varphi(\gamma_1(t))}{\varphi(\gamma_2(t))}\le\frac{\varphi(u)+\varphi(v_1)}{\varphi(u)}<\infty
\]
for all $t\geq 0$. 
Thus, 
\[
\limsup_{k\to\infty}M(\gamma_1(t_k)/\gamma_2(t_k))=\limsup_{k\to\infty}\left(\sup_{\varphi\in\Sigma_\xi^*}\frac{\varphi(\gamma_1(t_k))}{\varphi(\gamma_2(t_k))}\right)<\infty
\]
for all sequences $(t_k)_k$ with $t_k\to\infty$. Interchanging the roles of $\gamma_1$ and $\gamma_2$ yields an analogous result, from which we deduce that  
\[
\limsup_{k\to\infty}d_C(\gamma_1(t_k),\gamma_2(t_k))<\infty
\]
for all sequences $(t_k)_k$ with $t_k\to\infty$. These findings imply that for each $k\geq 1$ there exist $x_k$ on $\ell_1$ and $y_k$ on $\ell_2$ with $\sup_{k\geq 1}\|x_k-y_k\|<\infty$ such that $\|x_k\|\to\infty$ and $\|y_k\|\to\infty$ as $k\to\infty$. This implies that $\ell_1$ and $\ell_2$ are parallel.

As the extreme rays of the polyhedral cone $C$ are exposed, there exists a functional $\vartheta\in C^*$ such that $\vartheta(\mu v_2)=0$ for all $\mu\geq 0$  and $\vartheta(x)>0$ for all $x\in C\setminus\{\mu v_2\colon\mu\geq 0\}$. After scaling with an appropriate factor, we have $\vartheta\in\Sigma_\xi^*$.  Now it follows that
\[
\frac{\vartheta(\gamma_1(t))}{\vartheta(\gamma_2(t))}=\frac{e^{t}\vartheta(u)+e^{-t}\vartheta(v_1)}{e^{t}\vartheta(u)+e^{-t}\vartheta(v_2)}=\frac{\vartheta(u)+e^{-2t}\vartheta(v_1)}{\vartheta(u)}\to\infty
\]
as $t\to-\infty$, and hence  $M(\gamma_1(t)/\gamma_2(t))\to\infty$ as $t\to - \infty$. 
This, however, implies that 
\[
\lim_{t\to-\infty}d_T(\gamma_1(t),\gamma_2(t))=\infty
\] 
and hence  $\ell_1$ and $\ell_2$ are not parallel, which is absurd.  Thus, $C$ must be a simplicial cone.

Conversely, if $C$ is a simplicial cone in a, say $n$-dimensional vector space $X$, then there are linearly independent $v_1,\ldots,v_n\in X$ such that $C=\{\sum_{k=1}^n\alpha_k v_k\colon \alpha_k\geq 0 \mbox{ for }1\leq k\leq n\}$. 
The map $T:X\to\mathbb{R}^n$ given by $\sum_{k=1}^n\alpha_k x_k\mapsto (\alpha_1,\ldots,\alpha_k)$ is a bijective linear map with $T ( C ) =\mathbb{R}^n_+$.  and hence $T$ is an isometry of $(C^\circ,d_C)$ onto $((\mathbb{R}^n_+)^\circ,d_{
\mathbb{R}^n_+})$. Recall that $((\mathbb{R}^n_+)^\circ,d_{\mathbb{R}^n_+})$ is  isometric to $(\mathbb{R}^n,\|\cdot\|_\infty)$, see \cite[Proposition 2.2.1]{LNBook}. In fact, the reader can easily check that the coordinatewise log function is an isometry.   
\end{proof}

\section{Isometries on strictly convex cones}
In this section we will analyze the isometries of $(C^\circ,d_C)$ when $C$ is a closed strictly convex cone with nonempty interior  in a finite-dimensional vector space $V$. Recall that $C$ is {\em strictly convex} if for each linearly independent $x,y\in\partial C$ we have that 
\[
\lambda x+(1-\lambda)y\in C^\circ \mbox{\quad for all }0<\lambda <1.
\]

If $C$ is a closed cone with nonempty interior  in a normed space $V$ and $T\colon V\to V$ is an invertible linear map with $T ( C)=C$, then  $T$ is an isometry of $(C^\circ,d_C)$. 
Given a sclosed cone with nonempty interior  in a  finite-dimensional vector space $V$, we let $\mathrm{Aut}( C)=\{T\in\mathrm{GL(V)}\colon T( C)=C\}$ and we let $\mathrm{Isom}( C)$ be the set of maps $g\colon C^\circ\to C^\circ$ such that $g$ is an isometry of $(C^\circ,d_C)$.  So, $\mathrm{Aut}( C)$ is a subgroup of $\mathrm{Isom}( C)$. It is known \cite{Bo} that $\mathrm{Aut}( C)\neq \mathrm{Isom}( C)$, even if $C$ is a strictly convex cone.  Consider, for example,  the cone, $\Pi_2(\mathbb{R})$, of positive semi-definite matrices in the space of $2\times 2$ symmetric matrices. This is a $3$-dimensional, strictly convex, closed cone. In fact, $\Pi_2(\mathbb{R})$ is 
order-isomorphic with the 3-dimensional Lorentz cone, see \cite[p.\,44]{LNBook}. The map $h\colon  \Pi_2(\mathbb{R})^\circ\to \Pi_2(\mathbb{R})^\circ$ given by $h(A) = A^{-1}$ is an isometry under Thompson's metric, as $h$ is an order-reversing homogeneous degree $-1$ involution, see \cite[Corollary 2.1.5]{LNBook}. Obviously, $h\not\in \mathrm{Aut}( C)$. It turns out, however, that $h $ is projectively linear

\begin{definition}
A map $f\colon C^\circ\to C^\circ$ is {\em projectively linear} if there exists $T\in\mathrm{Aut}( C)$ such that for each $x\in C^\circ$, 
\[
f(x) = \lambda_x T(x)\mbox{\quad for some }\lambda_x>0.
\]
\end{definition}
Note that in the example above,  if 
\[
A=\left [ \begin{array}{cc} a & b \\ b & c\end{array}\right ]\in \Pi_2(\mathbb{R})^\circ,
\]
then 
\[
h(A) = A^{-1} = \frac{1}{\mathrm{det}\,(A)}\left [ \begin{array}{cc} c & -b \\- b & a\end{array}\right ],
\]
which shows that  $h$ is projectively linear. 

\begin{theorem}\label{thm:8.1}
If $C$ is a closed strictly convex cone with nonempty interior  in an  $n$-dimensional vector space $V$  and $n\geq 3$, then every $f\in\mathrm{Isom}(C )$ is projectively linear. 
\end{theorem}
\begin{proof}
Let $f\in\mathrm{Isom}( C)$. We will first show that $f$ maps type I geodesic lines to type I geodesic lines, and type II geodesic lines to type II geodesic lines. Suppose, by way of contradiction, that $\gamma$ is a type I geodesic line that is mapped to a type II geodesic line under $f$. Then $K=\mathrm{span}\{\gamma\}\cap C$ is closed 2-dimensional cone. So, by \cite[Lemma A.5.1]{LNBook} there exists linearly independent $u_0,v_0\in\partial C$ such that 
\[
K=\{\alpha u_0+\beta v_0\colon\alpha,\beta\geq 0\}.
\]
From Lemma \ref{lem:3.2}  we know that, after rescaling $u_0$ and $v_0$, we can write $\gamma$ as the image of 
\[
\gamma(t) =\frac{1}{2}(e^tu_0+e^{-t}v_0)
\]
where $t\in\mathbb{R}$. Let $x=\gamma(0)$ and $\varphi\in (C^*)^\circ$ with $\varphi (u_0)=1=\varphi (v_0)$. As $\dim C\geq 3$, $\Sigma_\varphi=\{v\in C\colon\varphi(v)=1\}$ is a compact convex set with $\dim \Sigma_\varphi \geq 2$. Thus, there exists a sequence $(u_k)_k$ in $\partial C$ with $\varphi(u_k)=1$  and $u_k\neq u_0$ for all $k\geq 1$ such  that $\|u_k- u_0\|\to 0$ as $k\to\infty$. Let $v_k\in\partial C$ be the point of intersection of the straight line through $u_k$ and $x$ and $\partial C$. So, $\|v_k-v_0\|\to 0$ as $k\to\infty$, since $x=\frac{1}{2}(u_0+v_0)$. For each $k\geq 1$ there exists $0<\alpha_k<1$ such that  $x =\alpha_k u_k+(1-\alpha_k)v_k$. Note that $\alpha_k\to 1/2$ as $k\to\infty$, and hence the geodesic paths,
\[
\gamma_k(t) = \alpha_ke^t u_k + (1-\alpha_k)e^{-t}v_k\mbox{\quad for }t\in\mathbb{R},
\]
are type I geodesics by Lemma \ref{lem:3.2} and $\gamma_k(t) \to\gamma(t) $ pointwise as $k\to\infty$. 

Now fix $y=\gamma(t)$ with $t\neq 0$ and consider the sequence $(y_k)_k$ with $y_k=\gamma_k(t)$. As the norm topology coincides with the Thompson's  metric topology on $C^\circ$, we have that $d_C(y_k,y)\to 0$ as $k\to\infty$. For $z\in C\setminus\{0\}$, write $[z]=z/\varphi(z)$. So, 
\[
\delta_C([y],[y_k])=\delta_C(y,y_k)\leq 2 d_C(y,y_k)\to 0
\]
as $k\to\infty$. 
Note that there is only one type II geodesic through $f(x)$, and hence $f(\gamma_k)$ must be a type I geodesic line  for all $k\geq 1$, as $f$ is an isometry. So, 
\begin{eqnarray*}
\delta_C([x],[y_k]) =\delta_C(x,y_k) & =  & 2d_C(x,y_k)\\
 & = & 2d_C(f(x),f(y_k)) =\delta_C([f(x)],[f(y_k)]).
\end{eqnarray*}
As $\gamma$ is mapped to a type II geodesic, the previous equality implies that $0<\delta_C([x],[y]) =\delta_C([f(x)],[f(y)])=0$, which is impossible. 
Thus $f$ maps type I geodesic lines  to type I geodesic lines. Also, $f$ has to map type II geodesic lines to type II geodesic lines, as otherwise $f^{-1}\in\mathrm{Isom}(C )$ maps a type I geodesic line to a type II geodesic line. 

Let $\Sigma_\varphi^\circ =\{x\in C^\circ\colon\varphi(x)=1\}$. 
Next we will show that $g\colon\Sigma_\varphi^\circ\to\Sigma_\varphi^\circ$ given by, 
\[
g(x) =\frac{f(x)}{\varphi(f(x))}\mbox{\quad for all }x\in\Sigma_\varphi^\circ,
\]
is an isometry under $\delta_C$. For $x$ and $y$ in $ \Sigma_\varphi^\circ$ distinct there exists a $\lambda > 0 $ such that  $x$ and $\lambda y$ lie on a type I geodesic in $(C^\circ, d_C)$. Write $\xi =\lambda y$. As  $M(x/\xi)=M(\xi/x)$, we have that 
\[
2d_C(x,\xi) =\delta_C(x,\xi) =\delta_C(x,y).
\]
Now using the fact that $f$ maps type I geodesic lines to type I geodesic lines we get that $M(f(x)/f(\xi))=M(f(\xi)/f(x))$. Also, as $f$ maps type II geodesic lines  to type II geodesic lines, 
\[
g(\xi) = \frac{f(\xi)}{\varphi(f(\xi))} = \frac{f(\lambda y)}{\varphi(f(\lambda y))} = \frac{f(y)}{\varphi(f(y))} =g(y).
\] 
Thus, 
\[
2d_C(x,\xi)= 2d_C(f(x),f(\xi)) =\delta_C(g(x),g(\xi)) =\delta_C(g(x),g(y)).
\]
It now follows that $g$ is an isometry under $\delta_C$. 
As $\Sigma_\varphi^\circ$ is a strictly  convex set, we deduce from \cite[Proposition 3]{dlH} that $f$ is a projectively linear map. 
\end{proof}
If $C$ is not strictly convex, $f\in \mathrm{Isom}( C)$ need not be projectively linear. Indeed, the map $(x,y,z)\mapsto (x,y, z^{-1})$ on the interior of the standard positive cone $\mathbb{R}^3_+=\{(x,y,z)\colon x,y,z\geq 0\}$ is an isometry under Thompson's metric but not projectively linear. It would be interesting to characterize those finite-dimensional  closed cones $C$ for which all Thompson's metric isometries are projectively linear. It would also be interesting to know for which cones $C$ we have $\mathrm{Aut}(C )=\mathrm{Isom}(C )$.

\subsection*{Acknowledgement} The authors are very grateful to Roger Nussbaum for sharing his proof of an earlier version of Theorem \ref{thm:5.1}. His remarks and insights have been very beneficial to us.


\begin{thebibliography}{1}

\bibitem{AGLN} M. Akian, S. Gaubert, B. Lemmens, and R.D. Nussbaum,
Iteration of order preserving subhomogeneous maps on a cone. 
\emph{Math. Proc. Cambridge Philos. Soc.} \textbf{140}(1), (2006), 157--176.

\bibitem{ACS} E. Andruchow, G. Corach, D. Stojanoff, Geometrical significance of L\"owner-Heinz inequality, {\em Proc. Amer. Math. Soc.} {\bf 128}(4), (2000), 1031-1037.

\bibitem{Ber} A. Bernig,  Hilbert geometry of polytopes. {\em Arch. Math. (Basel)} 
{\bf 92}(4), (2009), 314--324. 

\bibitem{Bi} G. Birkhoff, Extensions of Jentzsch's theorem.
\emph{Trans. Amer. Math. Soc.} \textbf{85}(1), (1957), 219--227.

\bibitem{Bo} A. Bosch\'e, Symmetric cones, the Hilbert and Thompson metrics, {\texttt
arXiv:1207.3214}, 2012. 

\bibitem{CV} B. Colbois and P. Verovic,  Hilbert domains that admit a quasi-isometric embedding into Euclidean space. {\em Adv. Geom.} {\bf 11}(3), (2011), 465--470.

\bibitem{Con} J. B. Conway, {\em A course in functional analysis.} 
Graduate Texts in Mathematics, {\bf 96}. Springer-Verlag, New York, 1990.

\bibitem{CM} G. Corach and A.L.  Maestripieri,
Differential and metrical structure of positive operators. 
{\em Positivity} {\bf 3}(4), (1999), 297--315.
 
  \bibitem{CPR} G. Corach, H.  Porta, and L. Recht, 
Convexity of the geodesic distance on spaces of positive operators. 
{\em Illinois J. Math.} {\bf 38}(1), (1994),  87--94.

\bibitem{dlH}  P. de la Harpe,  On Hilbert's metric for simplices. In {\em  Geometric group theory, Vol. 1 (Sussex, 1991)}, pp. 97Ð119, London Math. Soc. Lecture Note Ser., {\bf 181}, Cambridge Univ. Press, Cambridge, 1993. 

\bibitem{FK} J. Faraut and A. Kor\'anyi, {\em Analysis on Symmetric Cones}. 
Oxford Mathematical Monographs, Clarendon Press, Oxford, 1994.

\bibitem{FoK} T. Foertsch and A. Karlsson, Hilbert metrics and Minkowski norms. 
\emph{J. Geom.} \textbf{83}(1-2), (2005), 22--31.


\bibitem{HIR} D.H. Hyers, G. Isac, and  T.M. Rassias, {\em Topics in nonlinear analysis \& applications}. World Scientific Publishing Co., Inc., River Edge, NJ, 1997.


\bibitem{KN} A. Karlsson and G.A. Noskov, The Hilbert metric and Gromov hyperbolicity. \emph{Enseign. Math.\,(2)} \textbf{48}(1-2), (2002),  73--89.

\bibitem{Koe} M. Koecher, Positivit\"atsbereiche im $\mathbb{R}^n$. 
{\em Amer. J. Math.} {\bf 79}, (1957), 575Ð-596. 

\bibitem{LL1} J. Lawson and Y. Lim, Metric convexity of symmetric cones. {\em Osaka J. Math.} {\bf  44}(4),  (2007),  795--816. 

\bibitem{LNBook} B. Lemmens and R. Nussbaum, {\em Nonlinear Perron-Frobebius theory}. Cambridge Tracts in Mathematics {\bf 189}, Cambridge Univ. Press, Cambridge, 2012. 

\bibitem{Lim1} Y. Lim, Finsler metrics on symmetric cones. {\em Math. Ann.} {\bf 316},  (2000), 379--389. 

\bibitem{Lim2} Y. Lim, Hilbert's projective metric on Lorentz cones and Birkhoff formula for Lorentzian compressions. {\em Linear Algebra Appl.} {\bf 423}(2--3), (2007), 246--254.

\bibitem{Lim3} Y. Lim, Geometry of midpoint sets for Thompson's metric, {\em Linear Algebra Appl.}, to appear.

\bibitem{LP}  Y. Lim and M.  P\'alfia, Matrix power means and the Karcher mean. {\em J. Funct. Anal.} {\bf 262}(4), (2012), 1498--1514.

\bibitem{LW} C. Liverani and M. P. Wojtkowski: Generalization of the Hilbert metric to the space of positive definite matrices, {\em Pacific J. of Math.} {\bf 166}, (1994), 339--355. 

\bibitem{Mo} L. Moln\'ar. Thompson isometries of the space of invertible positive operators. {\em Proc. Amer. Math. Soc.} {\bf 137}, (2009), 3849--3859.

\bibitem{NS1} W. Noll and J.J. Sch\"affer, 
Orders, gauge, and distance in faceless linear cones; with examples relevant to continuum mechanics and relativity. 
{\em Arch. Rational Mech. Anal.} {\bf  66}(4),  (1977), 345--377. 

\bibitem{NS2} W. Noll and J.J. Sch\"affer, Order-isomorphisms in affine spaces. {\em Ann. Mat. Pura Appl. (4)} {\bf 117}, (1978), 243--262.

\bibitem{Nu} R.D. Nussbaum, Finsler structures for the part metric and Hilbert's projective metric and applications to ordinary differential equations. 
{\em Differential Integral Equations} {\bf 7}(5--6), (1994), 1649--1707. 

\bibitem{Nmem1} R.D. Nussbaum, 
Hilbert's projective metric and iterated nonlinear maps.
{\em Mem. Amer. Math. Soc.} {\bf 391},(1988), 1--137.

\bibitem{NW} R.D. Nussbaum and C. Walsh, A metric inequality for the Thompson and Hilbert geometries. {\em J. Inequal. Pure Appl. Math.} {\bf 5}(3),  (2004),  Article 54, 14 pp.

\bibitem{Pa} A. Papadopoulos, {\em  Metric Spaces, Convexity, and Nonpositive Curvature}. IRMA Lectures in Mathematics and Theoretical Physics 6. European Math. Soc. ZŸrich 2005. 

\bibitem{Rock} R.T. Rockafellar, \emph{Convex Analysis}, Princeton Landmarks in Mathematics, Princeton, N.J., 1997.


\bibitem{Tho} A.C. Thompson, On certain contraction mappings in a partially
ordered vector space. \emph{Proc. Amer. Math. Soc.} \textbf{14}, (1963),
438--443.

\bibitem{Vin} E.B. Vinberg, Homogeneous cones. {\em Soviet Math. Dokl.} {\bf 1}, (1960),  787--790.
\end{thebibliography}
\end{document}